\documentclass[11pt,a4paper]{article}
\usepackage{fullpage}
\usepackage{microtype}
\usepackage{graphicx}

\usepackage[colorlinks]{hyperref}
\usepackage[dvipsnames]{xcolor}
\hypersetup{
    linkcolor=red,
    citecolor=OliveGreen,
    filecolor=black,
    urlcolor=black,
}
\usepackage{caption}
\usepackage{upgreek}

\newcommand{\sP}{\mathsf{f}}

\newcommand{\Risk}{\mathsf{Risk}}

\newcommand{\beq}[1]{\begin{equation}\label{eq:#1}}
\newcommand{\eeq}{\end{equation}}

\renewcommand{\Re}{\operatorname{\mathsf{Re}}}
\renewcommand{\Im}{\operatorname{\mathsf{Im}}}
\usepackage{verbatim}
\usepackage{amsmath}
\usepackage{amsthm}
\usepackage{amssymb}
\usepackage{bm}
\usepackage{bbm}
\usepackage{dsfont}
\usepackage{comment}
\usepackage{multirow}
\usepackage{color}
\usepackage{tikz}
\usepackage{mathtools}
\usepackage{pifont}
\usepackage{amsmath}
\usepackage{enumerate, algorithm,algorithmic,caption,graphicx,color}

\usepackage{chngcntr}
\newcommand{\half}{ \mbox{\small$\frac{1}{2}$}}

\newcommand{\be}{\begin{eqnarray}}
\newcommand{\ee}[1]{\label{#1}\end{eqnarray}}

\newcommand{\ese}{\end{eqnarray*}}
\newcommand{\bse}{\begin{eqnarray*}}

\def\qed{\hfill$\square$}

\def\cJ{{\cal J}}

\def\cJ{{\cal J}}

\def\Diag{{\hbox{\rm Diag}}}

\def\E{{\mathbf{E}}}

\def\N{{\mathds{N}}}
\def\cN{{\mathcal{N}}}

\def\sF{{\mathsf{F}}}
\def\C{\mathds{C}}
\def\R{\mathds{R}}
\def\Z{\mathds{Z}}

\def\sR{\mathsf{R}}
\def\sL{\mathsf{r}}

\newcommand{\Prob}{\mathrm{Prob}}

\newcommand{\wh}[1]{{\widehat{#1}}}
\newcommand{\wt}[1]{{\widetilde{#1}}}

\def\Argmin{\mathop{\hbox{\rm Argmin}}}

\def\op{\textsf{op}}

\def\F{{\mathcal{F}}}
\def\T{{\mathds{T}}}
\def\D{{\mathds{D}}}

\def\vphi{\varphi}

\newcommand{\tr}{\textup{tr}}
\newtheorem{theorem}{Theorem}[section]
\newtheorem{proposition}{Proposition}[section]
\newtheorem{lemma}{Lemma}[section]

\newtheorem{definition}{Definition}[section]

\newtheorem{corollary}{Corollary}[section]
\newtheorem{conjecture}{Conjecture}[section]
\newtheorem{remark}{Remark}[section]

\newcommand{\veps}{\varepsilon}

\newcommand{\ones}{\mathds{1}}

\newcommand{\lang}{\left\langle}
\newcommand{\rang}{\right\rangle}

\renewcommand{\le}{\leqslant}
\renewcommand{\ge}{\geqslant}

\newcommand{\negphantom}[1]{\ifmmode\settowidth{\dimen0}{$#1$}\else\settowidth{\dimen0}{#1}\fi\hspace*{-\dimen0}}

\newcommand{\nvsps}{\vspace{-0.1cm}}

\allowdisplaybreaks

\usepackage{xfrac}
\usepackage{mathrsfs}
\usepackage{eucal}

\newcommand{\pinv}{\dagger}
\renewcommand{\top}{\mathsf{T}}
\newcommand{\htop}{\mathsf{H}}

\newcommand{\X}{X}
\newcommand{\XX}{\mathscr{X}}

\renewcommand{\S}{\mathsf{Supp}}
\newcommand{\Sc}{\mathsf{Supp}^{\mathsf{c}}}

\renewcommand{\E}{\mathsf{E}}

\newcommand{\Exp}{\mathds{E}}

\newcommand{\full}{\mathsf{full}}

\newcommand{\strue}{s_0}

\newcommand{\proofstep}[1]{$\boldsymbol{{#1}^o}$}

\newcommand{\bald}{\begin{aligned}}
\newcommand{\eald}{\end{aligned}}

\def\Dir{{\mathsf{Dir}}}
\def\Fej{{\mathsf{Fej}}}
\def\dir{{\mathsf{dir}}}
\def\fej{{\mathsf{fej}}}

\def\Filt{{\mathsf{Rest}}}
\def\Circ{{\mathsf{Circ}}}

\def\Hyp{\mathscr{H}}

\newcommand{\Posi}{\R_{+;1,\infty}}
\newcommand{\Real}{\R_{1,\infty}}
\newcommand{\Comp}{\C_{1,\infty}}

\newcommand{\Ext}{\mathsf{Ext}}

\newcommand{\poly}{\textup{poly}}

\usepackage{trimclip}

\makeatletter
\DeclareRobustCommand{\lasymp}{\lg@asymp{<}}
\DeclareRobustCommand{\gasymp}{\lg@asymp{>}}

\newcommand{\under@asymp}[1]{\clipbox{0pt 0pt 0pt {0.5\height}}{$\m@th#1\asymp$}}
\newcommand{\lg@asymp}[1]{\mathrel{\mathpalette\lg@asymp@{#1}}}
\newcommand{\lg@asymp@}[2]{%
  \vcenter{%
    \offinterlineskip
    \m@th
    \ialign{%
      \hfil##\hfil\cr
      $#1#2$\cr
      \under@asymp{#1}\cr
    }%
  }%
}
\makeatother

\newcommand{\odima}[1]{{#1}}

\begin{document}

\title{Near-Optimal and Tractable Estimation under Shift-Invariance}
\author{
Dmitrii M.~Ostrovskii\thanks{Georgia Institute of Technology, School of Mathematics \& H. Milton Stewart School of Industrial and Systems Engineering (ISyE), Atlanta, USA. Email: \texttt{ostrov@gatech.edu}.}
}

\maketitle

\begin{abstract}

How hard is it to estimate a discrete-time signal~$(x_{1}, ..., x_{n}) \in \mathds{C}^n$ satisfying an {\em unknown} linear recurrence relation of order~$s$ and observed in i.i.d.~complex Gaussian noise?
The class of all such signals is parametric but extremely rich: it contains all exponential polynomials over~$\mathds{C}$ with total degree~$s$, including harmonic oscillations with~$s$ arbitrary frequencies.
Geometrically, this class corresponds to the projection onto~$\C^{n}$ of the union of all shift-invariant subspaces of~$\mathds{C}^\mathds{Z}$ of dimension~$s$.
We show that the statistical complexity of this class, as measured by the squared minimax radius of the~$(1-\delta)$-confidence~$\ell_2$-ball, is nearly the same as for the class of~$s$-sparse signals, namely
$O\left(s\log(en) + \log(\delta^{-1})\right) \cdot \log^2(es) \cdot \log(en/s).$
Moreover, the corresponding near-minimax estimator is tractable, and it can be used to build a test statistic with a near-minimax detection threshold in the associated detection problem.
\odima{These statistical results rely upon a simple analytic observation: the interpretation of the Fourier coefficients of the Christoffel function of any shift-invariant subspace of~$\mathds{C}^\mathds{Z}$ as a reproducing filter with the smallest possible spectrum, in all~$\ell_p$-norms, $p \in [1,\infty]$, at once.}

\end{abstract}

\section{Introduction}
\label{sec:intro}

This paper is devoted to answering two questions posed by A.~Nemirovski in the 1990s~(\cite{nemirovski1992soviet,nemirovski_topics}):
\begin{quote}
\label{que:risk}
{\em 
What is the statistical complexity of estimating a solution to arbitrary, \underline{unknown} homogeneous linear difference equation of given order? 
Is there a tractable estimator?}
\end{quote}
To state them rigorously we need some notation.
Let~$\C(\Z)$ be the space of two-sided complex sequences, i.e.~$x = (x_t: t \in \Z)$ with~$x_t \in \C$. 
Let~$\Delta: \C(\Z) \to \C(\Z)$ be the delay (a.k.a.~lag, unit shift) operator, acting as~$(\Delta x)_t = x_{t-1}$.
A homogeneous linear difference equation of order~$s$ is
\begin{equation}
\label{eq:intro-ODE}
\sP(\Delta) x \equiv 0
\end{equation}
where~$\sP(\cdot)$ is a polynomial of degree~$s$ with~$\sP(0) = 1$; the r.h.s.~is the zero sequence.
Assume that a solution~$x \in \C(\Z)$ to such an equation is observed for~$t \in \Z: |t| \le n$, in Gaussian noise:
\begin{equation}
\label{eq:intro-observations}
y_t = x_t + \sigma \xi_{t}, \quad t \in [n]_\pm := \{-n, ..., n\},
\end{equation}
where~$\sigma$ is the noise level and~$\xi_{-n}, ..., \xi_{n} \sim \C\cN(0,1)$ are i.i.d.~standard complex Gaussians.
(Changing the observation domain from~$[n]_+ := \{1, ..., n\}$ to~$[n]_{\pm}$ is merely a matter of notation, but this change proves to be convenient for our purposes.)
Our goal is to recover~$x$ on~$[n]_{\pm}$.
As such, while a candidate estimator~$\wh x = \wh x(y_{-n}, ..., y_{n})$ is, formally, an element of~$\C(\Z)$, 
its performance is to be measured in terms of the mean-squared error (MSE) on~$[n]_{\pm}$, with some marginalization over the noise distribution.
To be more precise, fixing a desired confidence level~$1-\delta \in (0,1)$, for any set~$X \subseteq \C(\Z)$ we define the {\em worst-case} (over~$X$) {\em $\delta$-risk} of estimator~$\wh x$:
\[
\Risk_{n,\delta}(\wh x|X) := 
\inf 
\Bigg\{ 
\veps > 0: \; \Prob \Bigg(  \frac{1}{2n+1} \sum_{t \in [n]_{\pm}} |\wh x_t - x_t|^2 > \veps \Bigg) \le \delta \;\;\; \forall x \in X
\Bigg\}.
\]
That is,~$\Risk_{n,\delta}(\wh x|X)$ is the largest, over~$x \in X$, tight upper~$(1-\delta)$-confidence bound for the MSE.
Define the {\em minimax}~$\delta$-risk over~$X$ by minimizing~$\Risk_{n,\delta}(\wh x|X)$ over all possible estimators:
\begin{equation}
\label{def:minimax-risk}
\Risk_{n,\delta}^\star(X) = \inf_{\wh x: \C^{2n+1} \to X} \Risk_{n,\delta}(\wh x|X).
\end{equation}
(Hereinafter, we always assume measurability when necessary.)
Finally, define~$\XX_s$ as the union of the solution sets of all equations of the form~\eqref{eq:intro-ODE} and order~$s$, with all possible polynomials~$\sP$. 
Nemirovski's first question, in its most challenging version, amounts to computing the minimax risk~$\Risk_{n,\delta}^\star(\XX_s)$ and constructing an estimator~$\wh x^\star$ attaining it, that is a~{\em minimax estimator}.
More realistically, one may settle on estimating~$\Risk_{n,\delta}^\star(\XX_s)$ up to a polylogarithmic in~$n,s$ factor, thus exhibiting a {\em near-minimax} estimator.
Nemirovski's second question is whether one can construct a near-minimax estimator which is tractable: ideally, expressed as an optimal solution to a well-structured convex optimization problem, such as~a linear or semidefinite program~\cite{boyd2004convex}.\\

Before we move further, let us recall the geometric interpretation of equations of the form~\eqref{eq:intro-ODE}.
It is straightforward to show that, for any such equation with~$\deg(\sP) = s$, the solution set is an~$s$-dimensional $\Delta$-invariant linear subspace of~$\C(\Z)$; for brevity we shall call such subspaces {\em shift-invariant} and use the acronym SIS.
Conversely, any~$s$-dimensional SIS is the solution set for a unique equation of the form~\eqref{eq:intro-ODE}; see e.g.~\cite[Prop.~4.1.2]{ostrovskii2018adaptive}.
As the result,~$\XX_s$ is precisely the union of all shift-invariant subspaces of~$\C(\Z)$ of dimension~$s$.
In particular, since~$\XX_s$ contains the subspace of all polynomials of degree at most~$s-1$, i.e.~solutions to~\eqref{eq:intro-ODE} with~$\sP(z) = (1-z)^s$, 
we get a trivial lower bound under the assumption---always made henceforward---that~$2n+1 \ge s$:~\footnote{We write~$f \lasymp g$,~$f \gasymp g$,~$f \asymp g$, respectively, when~${f}/{g} \le C$,\,~${f}/{g} \ge c$,\;~$c \le {f}/{g} \le C$ for some constants~$C,c > 0$.}
\begin{equation*}
\Risk_{n,\delta}^\star(\XX_s) \gasymp \frac{\sigma^2}{2n+1} (s + \log(\delta^{-1})).
\end{equation*}
Indeed,~$\Risk_{n,\delta}^\star(\cdot)$ is nondecreasing in the set, and the minimax~$\delta$-risk on a subspace is known~\cite{mourtada2022exact}.

\paragraph{Baseline: super-resolution methods.}
One way to highlight the challenges presented by Nemirovski's questions is via their connection with Super-Resolution~\cite{candes2014towards} or ``spectral compressive sensing''~\cite{duarte2013spectral}.
Define~$\XX(\Omega)$, for~$\Omega \subseteq \C^{s}$, as the union of solution sets for equations of the form~\eqref{eq:intro-ODE} and such that the~$s$-tuple of the roots of~$\sP$, with~$\deg(\sP) = s$, belongs to~$\Omega$; in particular,~$\XX_s = \XX(\C^s)$.
Let~$\T$ be the unit circle, and~$\T_n$ be the set of~$(2n+1)^{\textup{st}}$ roots of unity:
\begin{equation}
\T_n := \left\{z \in \T: z^{2n+1} = 1 \right\}.
\end{equation}
Choosing~$\arg(\cdot)$ to be~$(-\pi, \pi]$-valued, so that~$|\arg\left({z_1}/{z_2} \right)|$ is the arc-length metric over~$\T$, define
\begin{equation}
\label{eq:intro-separated-thorus}
\T_{s,n} := \left\{ (z_1, ..., z_s) \in \T^s: \; j \ne k \;\Rightarrow\; \left|\arg\left(\frac{z_j}{z_k} \right)\right| \ge \frac{2\pi}{2n+1} \right\}.
\end{equation}
Note that for the set~$\T^{(s)} := {\T \choose s}$ of~$s$-tuples without replacement,~$\XX(\T^{(s)})$ is the set of harmonic osclillations with~$s$ arbitrary frequencies, i.e.~sequences of the form~$x_t = \sum_{k = 1}^s a_k  e^{i \omega_k t}$ for~$\omega_k \in \R$ and~$a_k \in \C$.
Replacing~$\T^{(s)}$ with~$\T^s$, i.e.~allowing~$p(\cdot)$ to have repeated roots, we get ``generalized harmonic oscillations,'' where each of~$h \le s$ harmonics is polynomially modulated, with total degree~$s-h$ of the modulating polynomials.
Next,~$\XX(\T_n^{(s)})$, with~$\T_n^{(s)} := {\T_n \choose s}$, is the set of harmonic oscillations with frequencies restricted to the uniform grid~$\{\frac{2\pi k}{2n+1}: k \in [n]_{\pm}\}$ of the discrete Fourier transform (DFT).
Finally,~$\XX(\T_{s,n})$ is the set of harmonic oscillations with frequences pairwise separated by~$\smash{\frac{2\pi}{2n+1}}$, the Abbe limit~\cite{candes2014towards}.
Since~$\smash{\T_n^{(s)} \subseteq \T_{s,n} \subseteq \T^{(s)} \subseteq \T^s \subseteq \C^s}$, 
\[
\Risk_{n,\delta}^\star(\XX(\T_{n}^{(s)})) \le \Risk_{n,\delta}^\star(\XX(\T_{s,n})) \le \Risk_{n,\delta}^\star(\XX(\T^{(s)})) = \Risk_{n,\delta}^\star(\XX(\T^{s})) \le \Risk_{n,\delta}^\star(\XX_s)
\]
where the equality holds by a straightforward compactness argument. 
Now, observe that, by the unitarity of DFT and the unitary invariance of noise distribution and squared loss, the problem of denoising over~$\XX({\T_n^{(s)}})$ is equivalent\footnote{In the natural sense; formally, the Le Cam distance~\cite{Lecam1986} between the corresponding statistical models is zero.} to that of denoising~$s$-sparse vectors in the Gaussian sequence model -- for which the minimax~$\delta$-risk is known up to a constant factor (see e.g.~\cite{verzelen2012minimax}):
\begin{equation}
\label{eq:intro-lower}
\Risk_{n,\delta}^\star(\XX(\T_n^{(s)})) \asymp \displaystyle \frac{\sigma^2}{2n+1} (s \log(en/s) + \log(\delta^{-1})).
\end{equation}
This gives a lower bound for~$\Risk_{n,\delta}^\star(\XX_s)$ with the classical~$\log\big({n \choose s}\big) \asymp s \log(en/s)$                                                                                                                    complexity term.
On the other hand, exploiting some ideas and tools from sparse recovery, semidefinite programming and the theory of nonnegative polynomials, the works~\cite{recht2} and~\cite{candes2013super} showed that
\begin{equation}
\label{eq:recht-bound}
\Risk_{4n,\delta}^\star(\XX(\T_{s,n})) \lasymp \frac{\sigma^2}{2n+1} s\log(e n \delta^{-1}).
\end{equation}
Note that~\eqref{eq:recht-bound} would match the lower bound~\eqref{eq:intro-lower} if one ignored the~$s$-fold inflation of the confidence term, and oversampling by a constant factor, i.e.~pairwise frequency separation of~$4$ DFT bins.
Moreover, the corresponding estimators are tractable -- computed via semidefinite programming.
However, these methods do not allow to go beyond the case of harmonic oscillations with well-separated frequencies. 
This is because they are based on the compressive sensing paradigm~\cite{duarte2013spectral}: if we define the moment map~$\Phi: L_1(\T) \to \C(\Z)$ such that
\[
[\Phi(\nu)]_t = \int_{z \in \T} z^t d \nu(z),
\]
the signal is estimated by~$\wh x = \Phi(\hat \nu)$, where
\begin{equation}
\label{eq:intro-AST}
\wh \nu \in \Argmin_{\nu \in L_2(\T)} \lambda \|\nu\|_{L_1(\T)} + \sum_{t \in [n]_{\pm}} |y_t - [\Phi(\nu)]_t|^2.
\end{equation}
For such estimators, obtaining ``fast''---order~$1/n$---rates for the MSE relies on assuming near-orthogonality of the dictionary elements comprising the signal, such as the restricted isometry property (RIP) or other conditions~\cite{vandegeer2009}, and using the dual problem of~\eqref{eq:intro-AST} for {support recovery} (here, recovery of the SIS corresponding to~$p$). 
Support recovery must be {exact} in the noiseless regime, i.e. when~$\sigma = 0$; 
yet, extensive simulations in~\cite{candes2014towards} show that the estimator~\eqref{eq:intro-AST} does not allow to recover the individual frequencies of~$x \in \XX(\T_{s,n})$~from~$(x_{-m}, ..., x_{m})$ if~$m \le n$.
Nonsurprisingly in this connection, we are not aware of any guarantee on~$\Risk_{n,\delta}^\star(\XX(\T_{s,n}))$ delivered by the Lasso-type estimator~\eqref{eq:intro-AST}. 

Meanwhile, a very different approach, proposed in its earliest version in the pioneering works~\cite{jn1-2009,jn2-2010} by Juditsky and Nemirovski, allows to upper-bound~$\Risk_{n,\delta}^\star(\XX_s)$ by~$\frac{\sigma^2}{2n+1}$ times a factor that is polynomial in~$s$,~$\log n$ and~$\log(\delta^{-1})$. 
We shall now give an overview this approach.

\paragraph{Going beyond~$\T_{s,n}$: adaptive estimators based on reproducing filters.}
In~\cite{jn1-2009} and~\cite{jn2-2010}, Juditsky and Nemirovski showed that, for~$n$ such that the right-hand side is at most~$\sigma^2$,
\[
\Risk_{n,\delta}^\star(\XX(\T^s)) \lasymp \frac{\sigma^2}{2n+1} \poly(s,\log(en\delta^{-1})).
\]
In~\cite{harchaoui2015adaptive}, the polynomial factor was improved through a more careful analysis of the same estimator.
Finally, \cite{harchaoui2019adaptive} managed to replace~$\T^s$ with~$\C^s$, while simultaneously further improving the polynomial factor, yet falling short of making it linear in~$s$.
Namely, combining~{\cite[Thm.~1]{harchaoui2019adaptive}} with a straightforward multiscale procedure described in Section~\ref{sec:l2-full}, one arrives at the bound
\begin{equation}
\label{eq:state-of-art-full}
\Risk_{n,\delta}^\star(\XX_s) \lasymp \frac{\sigma^2}{2n+1} \left( s^2 \log(e\delta^{-1}) + s\log n \right) \log(n/s). 
\end{equation}
As it turns out, improving this result any further, and especially breaking its quadratic dependence on~$s$, is a nontrivial challenge that has no less to do with approximation theory than with statistics.
To shed light on the nature of this challenge, we are now about to introduce the notion of {\em reproducing filters}.
First, let us define the discrete convolution operator on~$\C(\Z) \times \C(\Z)$,
\[
(u \ast v)_{t} = \sum_{\tau \in \Z} u_{\tau} v_{t-\tau}.
\]
This operator is bilinear and commutative. 
If we fix~$\vphi \in \C(\Z)$, then~$\vphi * x$ becomes a linear operator on~$\C(\Z)$, which corresponds to linear time-invariant filtering of~$x \in \C(\Z)$ with filter~$\vphi$.
\begin{definition}
\label{def:intro-reproducing}
{\em We call~$\vphi \in \C(\Z)$ {\em a reproducing filter} for~$X \subseteq \C(\Z)$ if~$\vphi * x = x$ for all~$x \in X$.}
\end{definition}
\noindent It is clear that, for any~$\vphi \in \C(\Z)$, the largest set~$X \subseteq \C(\Z)$ reproduced by~$\vphi$ is an SIS.
Now, let us define ``band-limited'' subspaces of~$\C(\Z)$, which correspond to finitely-supported sequences:
\[
\C_m(\Z) := \{x \in \C(\Z):\;\; |t| > m \;\Rightarrow\; x_{t} = 0 \}.
\]
Note that if~$\vphi \in \C_m(\Z)$, the convolution~$\vphi * x$ corresponds to {\em local} linear time-invariant filtering.
Now, if~$X$ is an~$s$-dimensional SIS, then there is a unique filter supported on~$\{1, ..., s\}$ and such that~$X$ is its maximal reproduced set, namely~$I-\sP(\Delta)$ where~$I$ is the identity operator on~$\C(\Z)$; see~\cite[Prop.~4.1.2]{ostrovskii2018adaptive}.
Yet, increasing the support we lose uniqueness of a reproducing filter, and it is natural to select a filter by minimizing certain norm.
It is, furthermore, natural to look for guarantees that are uniform over all SIS of a given dimension~$s$, and thus can only depend on~$s,m$, and the norm used.
As for the choice of norm(-s) to use, it turns out to be dictated by statistical considerations. 
Indeed, let~$\F_n: \C_n(\Z) \to \C^{2n+1}$ be the {discrete Fourier transform}:
\begin{equation*}
\label{def:intro-DFT}
(\F_n[x])_k = \frac{1}{\sqrt{2n+1}} \sum_{\tau \in [n]_{\pm}} \exp\left(-\frac{i2\pi k \tau}{2n+1}\right) x_{\vphantom{k,n}\tau}^{\vphantom{-\tau}}.
\end{equation*}
From~\cite{harchaoui2019adaptive} we extract the following two results whose combination leads to~\eqref{eq:state-of-art-full}, as we show below.
\begin{theorem}[{\cite{harchaoui2019adaptive}}]
\label{th:l2con-simp}
Assume~$X$ is an~$s$-dimensional SIS reproduced by~$\vphi^X \in \C_n(\Z)$ with~$n \gasymp s$, 
\begin{equation}
\label{eq:l2con-simp-premise}
\|\F_n[\vphi^X]\|_1 \le \frac{\sR_1}{\sqrt{2n+1}}, \quad \|\F_n[\vphi^X]\|_2 \le \frac{\sR_2}{\sqrt{2n+1}}. 
\end{equation}
Then the estimator~$\wh x = \wh \vphi * y$, where~$\wh \vphi = \wh \vphi(y)$ is an optimal solution to the optimization problem
\begin{equation}
\label{opt:l2con-simp}
\min_{\vphi\in \C_{n}(\Z)}
\left\{ 
\|\F_n[\vphi*y - y]\|_{2}: \quad \|\F_n[\vphi]\|_{1} \le \frac{\sR_1}{\sqrt{2n+1}}
\right\},
\end{equation}
admits the following guarantee for~$\delta$-risk:
$
\Risk_{n,\delta}(\wh x|X) \lasymp \frac{\sigma^2}{2n+1} 
\big(\sR_1 \log ({en}{\delta^{-1}}) + \sR_2^2 \log(e\delta^{-1}) + s \big).
$
\end{theorem}
\begin{proposition}[{\cite{harchaoui2019adaptive}}]
\label{prop:suboptimal-oracle}
Let~$n = 2m \ge 2s-2$. 
Any~$s$-dim.~SIS is reproduced by some~$\vphi \in \C_{n}(\Z)$:
\[
\|\F_{n}[\vphi]\|_2 \le \|\F_{n}[\vphi]\|_1 \le \frac{4s}{\sqrt{2n+1}}.
\]
\end{proposition}
\noindent Proposition~\ref{prop:suboptimal-oracle} allows to instantiate Theorem~\ref{th:l2con-simp} with~$\sR_1 = \sR_2 = s$, which implies the guarantee
\begin{equation}
\label{eq:state-of-art-core}
\Risk_{n,\delta}(\wh x|\XX_s) \lasymp \frac{\sigma^2}{2n+1} \left( s^2\log(e\delta^{-1}) + s\log n \right).
\end{equation}
This guarantee is attained with a tractable estimate, whose computation reduces to solving a convex program~\eqref{opt:l2con-simp}. 
The only issue preventing us from verifying the worst-case risk bound~\eqref{eq:state-of-art-full} is that the estimated ``slice''~$(\wh x_{t})_{t \in [n]_{\pm}}$ depends on the observations~$(y_\tau)_{\tau \in [2n]_{\pm}}$; simply put, we only estimate~$x_t$ on the ``core'' segment~$[n]_{\pm}$ of the whole observation domain~$[2n]_{\pm}$. 
As we discuss in Section~\ref{sec:l2-full}, there is a simple trick to convert such a ``core'' estimate into a full-domain one, at the price of the~$\log(n/s)$ factor: it suffices to compute the same estimator over a sequence of shrinking intervals whose pattern follows Cantor's ternary set. Thus, we convert~\eqref{eq:state-of-art-core} to~\eqref{eq:state-of-art-full}.







\vspace{-0.2cm}
\paragraph{Towards our results.}
The~$s^2$ term in the right-hand sides of~\eqref{eq:state-of-art-core} and~\eqref{eq:state-of-art-full} results from the~$\ell_2$-norm bound in Proposition~\ref{prop:suboptimal-oracle}, and would be improved to~$s$ if we strengthened Proposition~\ref{prop:suboptimal-oracle} by exhibiting a reproducing filter that satisfies~\eqref{eq:l2con-simp-premise} with~$\sR_1 \lasymp s$ and~$\sR_2 \lasymp s^{1/2}$. 
However, even this would not allow us to match the lower bound in~\eqref{eq:intro-lower} up to a logarithmic factor. This is because the confidence factor~$\log(\delta^{-1})$ in the risk guarantee of Theorem~\ref{th:l2con-simp} enters multiplicatively with~$s$-dependent factors~$\sR_1$ and~$\sR_2$.
There is, moreover, no hope to make either of these factors dimension independent. 
Indeed: for any reproducing filter on~$X$, one has~$\sR_1 \ge \sR_2$; on the other hand, we observe that~$\sR_2 \gasymp s^{1/2}$ by interpreting such a filter as a linear unbiased estimator on~$X$ whose expected squared loss~$\frac{\sigma^2}{2n+1}\sR_2^2$, and invoking the Gauss-Markov theorem. 
As such, in order to match the lower bound~\eqref{eq:intro-lower} one must not only strengthen Proposition~\ref{prop:suboptimal-oracle}, but also improve Theorem~\ref{th:l2con-simp} by decoupling the term~$\log(\delta^{-1})$ from both~$s$-dependent terms in~\eqref{eq:state-of-art-core}.

\subsection{Our results and organization of the paper}
The discussion in the above paragraph motivates our {\bf two key results}, which we now present.


\vspace{-0.1cm}
\begin{proposition}[\odima{Simplified version of Proposition~\ref{prop:hybrid-oracle}}]
\label{prop:hybrid-oracle-intro}
\odima{Assuming that~$s \in \N$ and~$n + 1 \ge s$, arbitrary~$s$-dimensional shift-invariant subspace of~$\C(\Z)$ admits a reproducing~$\vphi \in \C_{n}(\Z)$ with}\vspace{-0.1cm}
\begin{align}
\label{eq:oracle-norms-intro}
\|\F_{n}[\vphi]\|_p \le \frac{(\odima{2s})^{1/p}}{\sqrt{2n+1}} \quad\quad \forall p \in [1,+\infty].
\end{align}
\end{proposition}
\odima{Postponing any further discussion of this result until Section~\ref{sec:oracle}, let us only note here that the constant factor~$2$ multiplying~$s$ in the numerator is almost sharp: one can verify that~\eqref{eq:oracle-norms-intro} with~$1$ instead of~$2$ is attained on any subspace of {\em periodic} harmonic oscillations with~$s$ frequencies.}
\begin{theorem}[Version of Theorem~\ref{th:l2con-core} with generic constants]
\label{th:l2-core-intro}
Assume~$X$ is an~$s$-dimensional SIS reproduced by~$\vphi^X \in \C_n(\Z)$ with~$n \gasymp s$ and
\begin{equation}
\label{eq:l2-core-intro-premise}
\begin{aligned}
\|\F_n[\vphi^X]\|_1 \le \frac{\sR_1}{\sqrt{2n+1}}, \quad \|\F_n[\vphi^X]\|_2 &\le \frac{\sR_2}{\sqrt{2n+1}}, \quad \|\F_n[\vphi^X]\|_\infty \le \frac{\sR_\infty}{\sqrt{2n+1}},
\end{aligned}
\end{equation}
where~$\sR_1 \asymp s,$~$\sR_2 \asymp \sqrt{s},$ and~$\sR_\infty \asymp 1.$
Then~$\wh x = \wh \vphi * y$, where~$\wh \vphi = \wh \vphi(y)$ is an optimal solution to 
\begin{equation}
\label{opt:l2-core-intro}
\min_{\vphi\in \C_{n}(\Z)}
\left\{ 
\|\F_n[\vphi*y - y]\|_{2}: \quad \|\F_n[\vphi]\|_{1} \le \frac{\sR_1}{\sqrt{2n+1}}, \quad \|\F_n[\vphi]\|_{\infty} \le \frac{\sR_\infty}{\sqrt{2n+1}}
\right\},
\end{equation}
admits the following bound on the worst-case~$\delta$-risk:
\begin{equation}
\label{eq:l2-core-intro-risk}
\begin{aligned}
\Risk_{n,\delta}(\wh x|X) 
&\lasymp \frac{\sigma^2}{2n+1} \Big(s\log (en)  + \log(\delta^{-1})\Big) \log^2(es).
\end{aligned}
\end{equation}
%
\end{theorem}
\noindent
Note that the factor in the parentheses in~\eqref{eq:l2-core-intro-risk} can be read as
$
\sR_1 \log(n) + \sR_2^2 + \sR_{\infty}^2\log(\delta^{-1})
$
for~$\sR_1, \sR_2, \sR_\infty$ as in the premise of the theorem; thus, we indeed have managed to decouple~$\log(\delta^{-1})$ from both~$s$-dependent factors, as prescribed.
Combining Proposition~\ref{prop:hybrid-oracle-intro} and Theorem~\ref{th:l2-core-intro}, and noting that the estimator in~\eqref{opt:l2-core-intro} is pivotal to the subspace~$X$, we deduce that the bound~\eqref{eq:l2-core-intro-risk} extends to the worst-case $\delta$-risk~$\Risk_{n,\delta}(\wh x|\XX_s)$ of adaptive estimation over the union of {all}~$s$-dimensional SIS.
Finally, same as for~\eqref{eq:state-of-art-core}, we bound the minimax risk as
\begin{equation}
\label{eq:l2-core-intro-risk-minimax}
\begin{aligned}
\Risk_{n,\delta}^\star(\XX_s) 
&\lasymp \frac{\sigma^2}{2n+1} \Big(s\log (en)  + \log(\delta^{-1}) \Big) \log^2(es)  \log(en/s),
\end{aligned}
\end{equation}
by constructing a full-domain estimate in a multi-scale fashion; see Corollary~\ref{cor:l2-full} in Section~\ref{sec:l2-full} for the precise result.
Thus, the lower bound~\eqref{eq:intro-lower} is matched up to logarithmic factors, and our proposed near-minimax estimator is tractable -- computed via second-order cone programming.

We defer further discussion of these central results and intuition behind them to Sections~\ref{sec:oracle}--\ref{sec:l2-bound}.

\paragraph{Extensions.} 
In addition to those just discussed, we prove other results worth to be mentioned.
\begin{itemize}
\item[] 
{\bf One-sided recovery.}
In the second part of Section~\ref{sec:oracle}, we prove an analogue of Proposition~\ref{prop:hybrid-oracle-intro} where the reproducing filter is constrained to be causal (``one-sided'')---i.e.~ordinary, rather than Laurent, polynomial of degree~$s$. 
As it turns out, in this case the factors~$\sR_1$ and~$\sR_2$ degrade to~$s^2$ and~$s$, respectively; we show this to be optimal in the case of~$\sR_2$, and conjecture the bound on~$\sR_1$ to be optimal as well.
The statistical guarantee, i.e.~a counterpart of Theorem~\ref{th:l2-core-intro} where the estimated filter is constrained to be causal, deteriorates respectively. 
This result is rigorously formulated as Theorem~\ref{th:l2con-half} in Section~\ref{sec:l2-bound}.
\item[]
{\bf Signal detection.}
In Section~\ref{sec:l2-bound}, after discussing the full version of Theorem~\ref{th:l2-core-intro} (cf.~Theorem~\ref{th:l2con-core}) and constructing a near-minimax full-domain estimator based on it (Section~\ref{sec:l2-full}), we change our focus to the signal detection problem, first considered in~\cite{jn-2013}, where the goal is to discover the mere presence of~$x \in \XX_s$ in the Gaussian noise
In a nutshell, it turns out that the respective near-minimax detection threshold, measured in~$\ell_2$-norm, up to logarithmic factors matches the minimax detection threshold for {$s$-sparse signals} in Gaussian noise~\cite{ingster2010detection}. Simply put, it is almost no harder to {\em detect} on~$\XX_s$ than on~$\XX(\T_n^{(s)})$. 

\end{itemize}

\vspace{-0.2cm}
\paragraph{Paper organization.}
\odima{In Section~\ref{sec:oracle}, we prove and discuss Proposition~\ref{prop:hybrid-oracle-intro}, as well as a counterpart guarantee applicable to the one-sided (``causal") filtering scenario.}
Section~\ref{sec:l2-bound} is ``statistical:'' in it, we present statistical guarantees for estimators based on the analytical results presented in Section~\ref{sec:oracle}. 
We carry out the proofs of these guarantees in Section~\ref{sec:stat-proofs}. 
Some technical results on trigonometric interpolation and minimal-norm reproducing filters are deferred to appendices.
In particular, in Appendix~\ref{apx:dirichlet-and-fejer} we review the properties of the Dirichlet and Fej\'er kernels and establish the various bounds for summation of these kernels over equidistant grids;
in Appendix~\ref{apx:shift-inv} we exhibit the matching lower bound for the~$\ell_2$-norm of a causal reproducing filter.

\odima{
In the previous version of the paper, our proof of Proposition~\ref{prop:hybrid-oracle-intro} was erroneous. While the mistake can be corrected (which by itself is a nontrivial exercise in harmonic analysis), we instead chose to provide an alternative, more elegant construction using the Christoffel function.
}

\subsection{Notation}
\label{sec:notation}

\paragraph{Vectors, matrices.}
For~$m,n \in \N$, we index the entries of vectors in~$\C^n$ starting from~$0$; similarly for the rows and columns of matrices in~$\C^{m \times n}$.
For~$A \in \C^{m \times n}$, we denote with~$\bar A$ its entrywise conjugate (i.e.~$\bar A_{jk} = \overline{A_{jk}}$), with~$A^\top$ its transpose, with~$A^\htop = (\bar A)^\top$ its conjugate transpose.
We write~$\|\cdot\|_p$ and~$\lang \cdot, \cdot \rang$ for~$p$-norms and Hermitian dot product~$\langle a, b \rangle = a^\htop b$ on~$\C^n$.

\paragraph{Growth order.}
For two functions~$f,g > 0$ on the same domain, we write~$f \lasymp g$,~$f \gasymp g$,~$f \asymp g$ if there are some constants~$c,C > 0$ such that~${f}/{g} \le C$,~${f}/{g} \ge c$, and~$c \le {f}/{g} \le C$ respectively.

\paragraph{Sequences.}
We define the vector space of two-sided sequences~$\C(\Z) := \{(x_{t})_{t \in \Z}: x_t \in \C\}$.
For~$n \in \N$, let~$\C_n(\Z)$ be the subspace of~$\C(\Z)$ comprised of all~$x$ such that~$x_t = 0$ for~$|t| > n$. 
We let~$\|\cdot\|_p$,~$p \ge 1$, also be the~$p$-norm on~$\C(\Z)$, and~$\| \cdot \|_{n,p}$ be the seminorm on~$\C(\Z)$ as follows:
\[
\|x\|_{n,p} := \left( \sum_{t \in [n]_\pm} |x_t|^p \right)^{1/p}.
\] 
We define the Hermitian form~$\langle \cdot,\cdot\rangle_n$ on~$\C(\Z)$ by~$\langle u,v\rangle_n := \sum_{|t| \le n} \bar u_t v_t$; thus~$\|x\|_{n,2} = \sqrt{\langle x, x \rangle_n}$.

\paragraph{Convolution and~$z$-transform.}
The discrete convolution~$u \ast v$ of~$u, v \in \C(\Z)$ is defined by
\[
(u \ast v)_{t} = \sum_{\tau \in \Z} u_{\tau} v_{t-\tau}, \quad t \in \Z.
\]
We use the $z$-transform formalism, associating to~$u \in \C(\Z)$ the formal series~$u(z) := \sum_{\tau \in \Z} u_\tau z^{-\tau}$.
Clearly,~$[u*v](z) = u(z) v(z)$, so convolution is a commutative, associative, and bilinear operation. 
When viewed as a function on~$\T$, the~$z$-transform of~$u \in \C(\Z)$ is a Laurent series, or Laurent polynomial if~$u \in \C_n(\Z)$, and is also called the {discrete-time} Fourier transform (DTFT).

\paragraph{Discrete Fourier transform.}
For~$n \in \N$, we let~$\T_n$ be the set of all~$(2n+1)^{\textup{st}}$ roots of unity:
\begin{equation}
\label{def:DFT-grid}
\T_n := \left\{\chi_{k,n} := \exp\left(\frac{i2\pi k}{2n+1}\right),\;\; k \in [n]_{\pm} \right\}.
\end{equation}
We define the unitary {\em discrete Fourier transform} (DFT) operator~$\F_n: \C(\Z) \to \C^{2n+1}$ as follows:
\begin{equation*}
(\F_n[u])_k = (2n+1)^{-1/2} \sum_{\tau \in [n]_{\pm}} \chi_{k,n}^{-\tau} u_{\vphantom{k,n}\tau}^{\vphantom{-\tau}}.
\end{equation*}
(Here we use our convention of indexing~$\C^{2n+1}$ on~$\{0,...,2n\}$; note also that~$\chi_{k,n} = \chi_{k-2n-1,n}$.)
For~$u \in \C_n(\Z)$, this can be expressed succintly via~$z$-transform:
\begin{equation}
\label{def:DFT}
(\F_n[u])_k = (2n+1)^{-1/2} u(\chi_{k,n}^{-1}). 
\end{equation}
The restriction of~$\F_n$ on~$\C_n(\Z)$ is an isometry; its generalized inverse~$\F_n^{\,-1} : \C^{2n+1} \to \C_n(\Z)$ is
\begin{equation*}
\label{def:iDFT}
(\F_n^{\,-1}[a])_{\tau} =  (2n+1)^{-1/2} \ones\{\tau \in [n]_{\pm}\}  \sum_{k = 0}^{2n} \chi_{k,n}^{\tau} a_k^{\vphantom\tau}.
\end{equation*}
Furthermore, Parseval's identity states that~$\|u\|_{n,2} = \|\F_n[u]\|_2$ for all~$u \in \C(\Z)$. 
More generally,
\begin{equation}
\label{def:Parseval}
\lang u,v \rang_n = \langle \F_n[u], \F_n[v] \rangle \quad \forall u,v \in \C(\Z).
\end{equation}



\paragraph{Shift-invariant subspaces.}
For~$w_1, ..., w_s \in \C$, we let~$\X(w_1,...,w_s)$ be the solution set of the equation~\eqref{eq:intro-ODE} whose characteristic polynomial is 
$
\sP(z) = \prod_{k \in [s]} \left(1-{w_k}{z}\right),
$
i.e.~has~$w_1^{-1}, ..., w_s^{-1}$ as its roots.
Note that if all characteristic roots are distinct,~$X(w_1, ..., w_s)$ is the span of~$s$ complex exponentials~$w_1^t,..., w_s^t$, i.e.~$x_t = \sum_{k=1}^s c_k^{\vphantom t} w_k^t$ for~$c_1, ..., c_s \in \C$. 
Repeated roots result in polynomial modulation: if the root~$w_{k}^{-1}$ of~$\sP$ has multiplicity~$m_k$,~$\sum_{k \in [s]} m_k = s$, then~$X(w_1, ..., w_s)$ comprises all signals of the form~$x_t = \sum_{k \in [s]} q_k(t)^{\vphantom t} w_{k}^t$ where~$q_k$ is a polynomial of degree~$m_k - 1$.
\section{Reproducing filters with near-optimal norms}
\label{sec:oracle}



We begin this section by giving a detailed formulation of Proposition~\ref{prop:hybrid-oracle-intro} with concrete constants.

\begin{proposition}
\label{prop:hybrid-oracle}
Let~$n + 1 \ge s$. For arbitrary~$w_1, ..., w_s \in \C$, the shift-invariant subspace~$X = X(w_1, ..., w_s)$ admits a reproducing filter~$\vphi^X \in \C_{n}(\Z)$ such that
$\vphi^X(z) \in [0,1]$ on~$\T$, as well as
\begin{equation}
\label{eq:integral-bound}
\frac{1}{2\pi} \int_{-\pi}^\pi \vphi^X(e^{i\theta}) d\theta = \frac{s}{n+1}.
\end{equation}
As the result, the following inequalities hold simultaneously with the constant~$\smash{c_n = \frac{2n+1}{n+1} \le 2}$:
\begin{align}
\label{eq:oracle-norms}
\|\F_{n}[\vphi^X]\|_p &\le \frac{(c_n s)^{1/p}}{\sqrt{2n+1}} \quad \forall p \in [1,\infty].
\end{align}
Moreover, the certificate property holds if~$n \ge s$: one has~$\vphi^X < 1$ strictly on~$\T \setminus \{w_1^{-1}, \dots, w_s^{-1}\}$.
\end{proposition}

\begin{remark}
\label{rem:diagonal-case}
{\em Inequality~\eqref{eq:oracle-norms} is optimal up to a constant factor, in the sense that even without imposing~$\vphi(\cdot) \in [0,1]$ on~$\T$, the constant~$c_n$ cannot be decreased below~$1$. 
Indeed, fix arbitrary SIS of {\em periodic} on~$[n]_{\pm}$ harmonic oscillations, e.g.~$X = \{x \in \C(\Z): x_t = \sum_{k = 1}^s c_k \chi_{k,n}^t, \, c_1, ..., c_s \in \C\}.$
For any~$x \in X$ and~$\vphi \in \C_n(\Z)$, the ordinary convolution~$\vphi * x$ is equal to the cyclic convolution modulo~$2n+1$, whence 
$
(\F_n[\vphi * x]) = \sqrt{2n+1}\, (\F_n[\vphi])_k \, (\F_n[x])_k
$
by the DFT diagonalization property. 
Since~$\F_n$ is bijective,~$\vphi$ reproduces~$X$ if and only if~$(\F_n[\vphi])_k \equiv 1$ for all~$k \in \{1, ..., s\}$. 
Setting the remaining entries of~$\F_n[\vphi]$ to zero minimizes the norms~$\|\F_n[\vphi]\|_p$ for all~$p \in [1,+\infty]$, and the corresponding filter~$\vphi^* \in \C_n(\Z)$ satisfies
$\| \F_n[\vphi^*]\|_p \sqrt{2n+1} \le {s^{1/p}}$ granted~$2n+1 \ge s$.}
\end{remark}
\vspace{-0.3cm}
\odima{
\begin{remark}
{\em 
Note that in the previous example, the reproducing filter~$\vphi^*$ turned out to be nonnegative on the grid~$\T_n$, but not on~$\T$. If we add the latter requirement, then the constant~$c_n$ in Proposition~\ref{prop:hybrid-oracle} is sharp; in particular, it is attained by~$\psi = (\phi^*)^2$ from the previous example. 
We leave the proof as an exercise to the reader (it essentially reduces to the Fej\'er-Riesz theorem).
}
\end{remark}
\begin{remark}
{\em 
By the Fej\' er-Riesz theorem, there exists a polynomial~$p(z) = \sum_{\tau = 0}^n p_\tau z^\tau$ (that can be identified with a {\em causal} filter) with complex coefficients and roots outside the unit disk, such that~$\vphi^X(z) = |p^X(z)|^2$ over~$\T$. 
By Proposition~\ref{prop:hybrid-oracle}, this polynomial satisfies~$|p(w_k^{-1})| = 1$ and~$|p(z)| \le 1$ on~$\T$, where the inequality is strict on~$\T \setminus \{w_1^{-1}, \dots, w_s^{-1}\}$ once~$n \ge s$.
At first glance, this might seem to contradict the well-known results (see e.g.~\cite{moitra2015super}) on the existence certificate polynomials, namely that such polynomials only exist once the angular separation of~$\{w_1^{-1}, \dots, w_s^{-1}\}$---assuming these points are on the circle---exceeds the Abbe limit~$\frac{2\pi}{2n+1}$; moreover, this seems to be at odds with our own results in Section~\ref{sec:oracle-causal}, giving an~$\smash{O(\frac{s}{\sqrt{n+1}})}$ sharp bound for the~$\ell_2$-norm of a causal reproducing filter. The catch here is that the causal factor~$p(\cdot)$ interpolates~$1$ in the characteristic roots only {\em up to a phase}, whereas for it to be reproducing one would need~$p(w_k^{-1}) = 1$. 
This corresponds to the fact that the trigonometric polynomial~$\vphi^X$ interpolates~$1$ in the characteristic roots; meanwhile, the trigonometric polynomial interpolating an {\em arbitrary sign pattern} at~$\{w_1^{-1}, \dots, w_s^{-1}\}$, while remaining~$\in (-1,1)$ elsewhere on~$\T$, is only possible under the separation assumption (as immediately follows from the Bernstein theorem).
}
\end{remark}}

\odima{
Finally, for the reference we recall two elementary results on the existence of reproducing filters for shift-invariant subspaces, used in prior work (see e.g.~\cite{harchaoui2019adaptive}) to establish Proposition~\ref{prop:suboptimal-oracle} and Theorem~\ref{th:l2-core-intro}. 
While none of these is going to be employed in the proposed construction, understanding them is a good warm-up exercise, so we provide them along with the short proofs.
}
\vspace{-0.2cm}

\begin{proposition}[cf.~{\cite[Prop.~2]{harchaoui2019adaptive}}]
\label{prop:projector}
Any shift-invariant subspace~$X$ of~$\C(\Z)$ of dimension~$s \ge 1$ admits a reproducing filter~$\phi \in \C_m(\Z)$, for arbitrary~$m \ge s-1$, such that
$
\|\phi\|_2^2 \le 2s/(2m+1).
$
\end{proposition}
\begin{proof}
Let~$X_m$ be the image of~$X$ under the slicing map~$x \in \C(\Z) \mapsto (x_0, ..., x_m) \in \C^{m+1}$. Clearly,~$X_m$ is a subspace of~$\C^{m+1}$ with~$\dim(X_m) \le s$.
Now, let~$\Pi_{m} \in \C^{(m+1) \times (m+1)}$ be the coordinate matrix of the projector on~$X_m$, in the canonical basis. 
Its squared Frobenius norm is~$\le s$, so~$\Pi_m$ has a row~$w \in \C^{m+1}$ such that~$\|w\|_2^2 \le \frac{s}{m+1} \le \frac{2s}{2m+1}$.
If the position of~$w$ in~$\Pi_m$ is~$t_0 \in \{0,...,m\}$, then~$x_{t_0} = \sum_{\tau = 0}^m w_\tau x_{\tau}$, so by time-reversing, shifting, and zero-padding~$w$~we get~$\phi \in \C_m(\Z)$ such that~$\|\phi\|_2 = \|w\|_2$ and~$x_{t_0} = \sum_{\tau \in [m]_{\pm}} \phi_\tau x_{t_0-\tau}$. 
Finally, the shift-invariance of~$X$ allows to replace~$t_0$ with arbitrary~$t \in \Z$. 
\end{proof}

\begin{proposition}[{\cite[Prop.~3]{harchaoui2015adaptive}}]
\label{prop:convolution}
If~$X \subseteq \C(\Z)$ is reproduced by~$\phi \in  \C_m(\Z)$:~$\| \phi \|_2^2 \le r^2/(2m+1)$,
then~$X$ is reproduced by the filter~$\phi^{2} := \phi * \phi \in \C_{2m}(\Z)$ for which
$
\|\F_{2m} [\phi^{2}]\|_1 \le 2r^2/\sqrt{4m+1}.
$
\end{proposition}
\begin{proof}
That~$\phi^2$ is reproducing is due to the fact that~$1 - \phi^2(z) = (1 - \phi(z)) (1 + \phi(z))$. 
On the other hand, the norm can be estimated via Parseval's identity and the fact that~$\phi \in \C_m(\Z)$:
\begin{align*}
\|\F_{2m} [\phi^{2}]\|_1 \sqrt{4m+1}
\stackrel{\eqref{def:DFT}}{=} 
    \sum_{z \in \T_{2m}} \left| \phi^2(z) \right| 
\stackrel{\eqref{def:DFT}}{=} 
	(4m+1) \|\F_{2m} [\phi]\|_2^2 
=  (4m+1) \|\phi\|_{2}^2 
\le 
   2r^2.
\tag*{\qedhere}
\end{align*}
\end{proof}
%
\begin{remark} 
{\em Proposition~\ref{prop:convolution} can be generalized to repeated autoconvolution~$\phi^k$. 
While such a result is not needed to prove Proposition~\ref{prop:hybrid-oracle}, we give it in Appendix~\ref{apx:convolution-powers} for completeness.}
\end{remark}
%
%
%
%
%
We shall now construct the reproducing filter whose existence is claimed in Proposition~\ref{prop:hybrid-oracle}. 

\vspace{-0.2cm}
\odima{
\subsection{Proof of Proposition~\ref{prop:hybrid-oracle}: construction of the reproducing filter}
For~$z \in \C$, let~$v_n(z) = [1\; z\; \cdots\; z^{n}]^\top \in \C^{n+1}$ be the associated moment vector. 
Moreover, let~$\smash{\Pi_n \in \C^{(n+1) \times (n+1)}}$ be the projector on~$\smash{X_n = \textup{span}\{v_n(w_1), \dots, v_n(w_s)\}}$, i.e. the ``slice'' of~$X$ (cf. the matrix~$\Pi_m$ in the proof of Proposition~\ref{prop:projector}).
Consider the following Laurent polynomial:\,\footnote{\odima{According to F.~Nazarov (private communication), this proof is ``from The Book and in the spirit of T.~Erdélyi.''}}
\begin{equation}
\label{eq:christoffel}
\vphi^X(z) = \frac{1}{n+1} v_n(z^{-1})^\top \Pi_n v_n(z).
\end{equation}
Clearly,~$\vphi^X \in \C_n(\Z)$ as a linear combination of the monomials~$1, z, \dots, z^n$ and their reciprocals. 
Furthermore, for~$z = w_k$ with~$k \in [s]$, we have that~$v_n(z)$ belongs to~$X_n$, which guarantees that
\[
\vphi^X(w_k) = \frac{1}{n+1} v_n((w_k)^{-1})^\top v_n(w_k) = 1.
\]
In other words,~$1-\vphi(z)$ vanishes at the reciprocal roots~$w_k = z_k^{-1}$ of the characteristic polynomial of~$X$. 
Here, we did {\em not} use that~$\smash{\frac{1}{\sqrt{n+1}}v(z)}$ is a unit vector (which holds only on~$\T$); instead, we only used the algebraic structure of moment vectors. 
Now, for any~$z \in \T$ we have~$z^{-1} = \bar z$, thus
\[
\vphi^X(z) = \frac{1}{n+1} v_n(z)^\htop \Pi_n v_n(z) = \frac{1}{n+1} \|\Pi_n v_n(z)\|_2^2 \quad \forall z \in \T.
\]
This implies that~$\vphi^X(z) \in [0,1]$ for any~$z \in \T$. Finally, to establish~\eqref{eq:integral-bound} we use the trace trick:
\[
\begin{aligned}
\frac{1}{2\pi} \int_{-\pi}^\pi \vphi^X(e^{i\theta}) d\theta 
&= \frac{1}{n+1} \int_{-\pi}^\pi \tr \left(\Pi_n v_n(e^{i\theta}) v_n(e^{i\theta})^\htop \right) d\theta \\
&= \frac{1}{n+1} \tr \left(\Pi_n \int_{-\pi}^\pi  v_n(e^{i\theta}) v_n(e^{i\theta})^\htop d\theta \right) 
= \frac{1}{n+1} \tr \left(\Pi_n \right) 
= \frac{s}{n+1}.
\end{aligned}
\]
We used that~$\int_{-\pi}^{\pi} e^{ik\theta} d\theta= 0$ for~$k \in \Z \setminus \{0\}$. (Note that any distinct~$w_1, \dots, w_s$ deliver equality.)
Finally, since~$\vphi^X = |\vphi^X|$ on~$\T$ and due to the cyclic property of the roots of unity, it holds that
\[
\|\F_n[\vphi^X]\|_1 
= \frac{1}{\sqrt{2n+1}} \sum_{z \in \T_n} \vphi^X(z) 
= \sqrt{2n+1} \vphi^X(1)
= \frac{\sqrt{2n+1}}{2\pi} \int_{-\pi}^\pi \vphi^X(e^{i\theta}) d\theta
\le \frac{s\sqrt{2n+1}}{n+1}.
\]
This verifies~\eqref{eq:oracle-norms} for~$p = 1$, whence the case of any~$p \in (1,\infty)$ follows by H\"older's inequality.
\qed
\begin{remark}
{\em 
Note that the reproducing filter in~\eqref{eq:christoffel} can be interpreted as the average of those obtained from the rows of~$\Pi_n$, in contrast to the one row used in Proposition~\ref{prop:projector}. 
}
\end{remark}
We note that in the theory of  orthogonal polynomials, the construction in~\eqref{eq:christoffel} is known as the Christoffel function. Recently, it found numerous applications in statistical theory and optimization~\cite{lasserre2024christoffel}. Our interpretation and application of the Christoffel function seem to be new.
}
\vspace{-0.2cm}
\subsection{One-sided reproducing filters for quasi-stable SIS}
\label{sec:oracle-causal}

We now establish an analogue of Proposition~\ref{prop:hybrid-oracle} under the constraint that the filter is {one-sided}.\vspace{-0.3cm}
%
\paragraph{Additional notation.}
Let~$\C_n^+(\Z)$ be the set of sequences supported within~$[n]_+ := \{0, ..., n\}$, that is~$\vphi \in \C_n^{+}(\Z)$ if and only if~$\vphi$ is a polynomial with~$\deg(\vphi)\le n$.
Note that~$u \in \C_{2n}^+(\Z)$ if and only if~$\Delta^{-n}u \in \C_{n}(\Z)$. 
We succintly define the one-sided DFT operator~$\F_{2n}^+: \C(\Z) \to \C^{2n+1}$,
\[
\F_{2n}^{+}[u] := \F_n^{\hphantom+} [\Delta^{-n} u].
\]
%

It is well-known that for~$\X(w_1,.., w_s)$ to admit a sequence of reproducing filters~$\vphi \in \C_n^+(\Z)$ whose norms are vanishing as~$n \to \infty$, one must have quasi-stability, i.e.~$w_j \in \D$~$\forall j \in \{1, ..., s\}$ where $\D$ is the closed unit disk. (This is clear for~$s = 1$: if $|w| > 1$ and~$\phi = (\phi_1, ..., \phi_n) \in \C^n$ is such that~$\sum_{k = 1}^n \bar \phi_k w^{-k} = 1$, then~$\| \phi \|_2 \ge \sqrt{1-|w|^{-2}}\sum_{k=1}^n \bar \phi_k w^{-k} \ge \sqrt{1-|w|^{-2}} > 0$ for all~$n$.)
As such, in the context of one-sided reproducing filters, we require that SIS is a quasi-stable one.

\begin{proposition}
\label{prop:hybrid-oracle-causal}
For some constants~$c_0,c_1 > 0$, the following holds. 
For all~$n \ge c_0 s^2 \log(2s)$, any SIS~$\X(w_1, ..., w_s)$ with~$w_1, ..., w_s \in \D$ admits a reproducing filter~$\vphi^X_+ \in \C_{2n}^+(\Z)$ such that
\begin{align}
\label{eq:oracle-norms-causal}
\|\F_{2n}^+[\vphi^X_+]\|_p \le \frac{c_1(s^2 \log(en))^{1/p}}{\sqrt{2n+1}} \quad \forall p \in [1,\infty].
\end{align}
\end{proposition}
Compared to the situation of Proposition~\ref{prop:hybrid-oracle}, the scaling factor~$\sR_p$ for the~$\ell_p$-norm degrades from~$s^{1/p}$ to~$s^{2/p}$ times a logarithmic factor.
This deterioration is due to the fact that we cannot use the filter~$\phi$ from Proposition~\ref{prop:projector} anymore, since this filter is two-sided by construction. (Which, in turn, is because in the pigeonhole trick employed to prove Proposition~\ref{prop:projector}, one does not know {\em which} row of the projector~$\Pi_m$ has minimal norm.)
Instead, Proposition~\ref{prop:hybrid-oracle-causal} is immediately implied by the following, more delicate result (proved by complex-analytic methods).
\vspace{-0.4cm}
\odima{
\begin{proposition}[{\cite[Prop.~4.5]{harchaoui2019adaptive}}]
\label{prop:causal-l2-oracle}
For~$m \gasymp s^2 \log(es)$, any SIS~$X(w_1,...,w_s)$ with~$w_1, ..., w_s \in \D$ admits a reproducing filter~$\phi_+ \in \C_{2m}^+(\Z)$ such that
\[
\|\phi_+\|_2^2 \lasymp \frac{40 (s+2)^2 \log(8ms)}{2m+1} 
\quad \text{and} \quad
\max_{z \in \T}|\phi_+(z)| \le 2.25.
\]
\end{proposition}
\begin{remark}
{\em 
In~\cite{harchaoui2019adaptive}, this result was stated for generalized harmonic oscillations, i.e.~$X(w_1, ..., w_s)$ with~$w_1, ..., w_s \in \T$, but the proof admits the quasi-stable case~$w_1, ..., w_s \in \D$.
(The reason~\cite{harchaoui2019adaptive} gave a restricted result is that in the context of subsequent statistical applications, they require~$X(w_1, ..., w_s)$ to be reproduced both by the ``past-looking" one-sided filter~$\vphi(\Delta) \in \C_n^+(\Z)$ and its ``future-looking" counterpart~$\psi(\Delta) = \vphi(\Delta^{-1})$.)
In addition, the bound for $\max_{z \in \T}|\phi_+(z)|$, while not stated explicitly in~\cite[Prop.~4.5]{harchaoui2019adaptive}, immediately follows from equations~(80)--(83) therein.
}
\end{remark}
}
\noindent
Proposition~\ref{prop:hybrid-oracle-causal} follows from Proposition~\ref{prop:causal-l2-oracle} by taking~$\vphi_+^{X}(z) = \phi_+^2(z)$ and adjusting the constants.
In contrast to Proposition~\ref{prop:projector}, the proof of Proposition~\ref{prop:causal-l2-oracle} is technical and relies on complex analysis.
Moreover, in Appendix~\ref{apx:one-sided} we show that, in the asymptoric regime~$m \to \infty$, this result is optimal up to a logarithmic factor.
Specifically, we show that~$m\| \phi_{+}^{*,m} \|_2^2 \to s^2$ for any sequence of~$\phi_{+}^{*,m} \in \C_m^+(\Z)$ of minimal-norm reproducing filters on the subspace of~$s$-degree polynomials.

\section{Statistical guarantees}
\label{sec:l2-bound}

\paragraph{Roadmap.}
In this section, we carry out the statistical ``program" announced in Section~\ref{sec:intro}. Namely, in Section~\ref{sec:l2-core}, we formulate and discuss the unabridged version of Theorem~\ref{th:l2-core-intro}, i.e.~a guarantee on the MSE of estimating the signal ``slice''~$(x_t, t \in [n]_{\pm})$ with an estimate of the form~$\wh x = \wh \vphi * y$, where the data-driven~$\wh \vphi \in \C_n(\Z)$ is built from the observations~$(y_t, t \in [2n]_{\pm})$. In Section~\ref{sec:l2-causal}, we make a brief detour to the ``one-sided'' setting, constructing an estimate~$\wh x = \wh\vphi * y$ that uses a {\em causal} data-driven filter~$\wh \vphi \in \C_{2n}^+(\Z)$; in a sense, Section~\ref{sec:l2-causal} stands in the same position towards Section~\ref{sec:l2-core} as Section~\ref{sec:oracle-causal} towards \odima{the preceding part of Section~\ref{sec:oracle}.}  In Section~\ref{sec:l2-full}, we revisit the ``two-sided'' guarantee of Section~\ref{sec:l2-core} and use it, together with a multiscale procedure, to build an estimate of~$(x_t, t \in [2n]_{\pm})$. 
Finally, in Section~\ref{sec:detection}, we focus on the associated {\em signal detection} problem, where the goal is to detect the presence of~$x \in \XX_s$ in the noise. 
To streamline the exposition, we postpone the proofs of all results to be presented in this section to Section~\ref{sec:stat-proofs}.

\subsection{Estimation ``on the core''}
\label{sec:l2-core}


Recalling \odima{Theorem~\ref{th:l2con-simp}}, consider the estimate~$\wh x_t = (\wh \vphi * y)_t$ with the following data-driven filter:
\begin{equation}
\label{opt:l2con-core}
\wh\vphi \in \Argmin_{\vphi\in \C_{n}(\Z)}
\left\{ 
\|\vphi*y - y\|_{n,2}^2:\;\;
\max\{\|\F_{n}[\vphi]\|_{1},  \odima{2s}\|\F_{n}[\vphi]\|_{\infty} \} \le \frac{\odima{2s}}{\sqrt{2n+1}}
\right\}.
\end{equation}
In plain words,~\eqref{opt:l2con-core} corresponds to minimizing the~$\ell_2$-norm residual of a convolution-type estimate~$\vphi * y$ under the~$\ell_1$- and~$\ell_\infty$-norm constraints on the filter DFT matching the bounds~\eqref{eq:oracle-norms} for the reproducing filter (cf.~Proposition~\ref{prop:hybrid-oracle}).
Same as in the proof of Proposition~\ref{prop:hybrid-oracle}, it suffices to impose only the constraints on the~$\ell_1$- and~$\ell_\infty$-norm; any other~$\ell_p$-norm can then be controlled via H\"older's inequality, and in our statistical analysis we  use the bound on the~$\ell_2$-norm as well.


\begin{remark}
\label{rem:l2con-computation}
{\em Observe that~\eqref{opt:l2con-core} is a well-structured convex program that can be efficiently solved by adapting first-order methods developed in~\cite{ostrovskii2018efficient} for the purpose of solving~\eqref{opt:l2con-simp}; in particular, the gradient oracle can be implemented via Fast Fourier Transform. 
In order to do so, one has to implement the proximal mapping for the complex version of the polytopal ball~\cite{deza2021polytopal}, i.e.~the intersection of a complex~$\ell_1$-ball and complex~$\ell_\infty$-ball. 
This can be done, e.g., by running a variant of the alternating projections method~\cite{bauschke1993convergence,bauschke1994dykstra,bauschke1996projection,grigoriadis2000alternating,lewis2008alternating} or extending the algorithm of~\cite{yu2012efficient} for efficient projection onto the intersection of norm balls to the case of a complex space.}
\end{remark}

Next we characterize the mean-squared error of the estimate in~\eqref{opt:l2con-core} on the ``core'' interval.
\begin{theorem}
\label{th:l2con-core}
In the assumptions of Proposition~\ref{prop:hybrid-oracle}, an optimal solution~$\hat\vphi \in \C_n(\Z)$ to~\eqref{opt:l2con-core} satisfies, for arbitrary~$\delta \in (0,1)$, the following inequality with probability at least $1-\delta$: 
\begin{equation}
\label{eq:l2-core}
\left\|\wh\vphi * y - x\right\|_{n,2}^2 
\le 
\odima{36}\,\sigma^2 \left( \odima{s}\log(2n+1) + \odima{2s} + \log(\delta^{-1}) \right) \log^2(e^4 s).
\end{equation}
\end{theorem}
\paragraph{Discussion.}
The result we have just stated is near-optimal:
as mentioned in the discussion following Theorem~\ref{th:l2con-simp}, neglecting the~$\log^2(s)$ factors, the right-hand side of~\eqref{eq:l2-core} matches the lower bound attained on the class of {\em periodic} harmonic oscillations, i.e.~signals with {\em sparse} DFT. 
This observation can be interpreted as the statistical counterpart of Remark~\ref{rem:diagonal-case}: the subclass~$\XX(\T_n^{(s)})$ is nearly worst-case in~$\XX_s$, from the statistical viewpoint as well.

Let us now make additional, more technical remarks regarding Theorem~\ref{th:l2con-core} and its proof.

Firstly, the leading term~$s \log(2n+1)$ in~\eqref{eq:l2-core} appears as the result of bounding the suprema of two stochastic processes with subexponential tails that arise in our error decomposition, namely~$\sup_{\vphi \in \Phi}\langle \xi, \vphi * \xi \rangle_n$ and~$\sup_{\vphi \in \Phi} \|\vphi * \xi \|_{n,2}^2$, where~$\Phi$ is the feasible set in~\eqref{opt:l2con-core}. 
Our control of these suprema, cf.~\eqref{eq:detection-II-term}--\eqref{eq:rand-convo-term-bound} and \eqref{eq:krein-milman-quad}--\eqref{eq:l2con-convo-quad-bound} respectively, relies on the Krein-Milman theorem and union bound over the extreme points of the set~$\{u \in \C^{2n+1}: \max\{\|u\|_{1}, s\|u\|_{\infty} \} \le s\}$, which are~$s$-sparse binary vectors with arbitrarily shifted phases. 
Factoring out the phases leaves us with~$N_s = {2n+1 \choose s}$ extreme points, so the union bound results in~$\log{N_s} \asymp s \log(2n+1)$ overhead.

Secondly, using Proposition~\ref{prop:hybrid-oracle} with Theorem~\ref{th:l2con-simp} also give the guarantee
$
O(\frac{s\sigma^2}{2n+1}\log(en\delta^{-1}))
$
for the~$\ell_1$-constrained estimate~\eqref{opt:l2con-simp}. 
The bound~\eqref{eq:l2-core} improves the confidence-dependent term from~$s\log(\delta^{-1})$ to~$\log^2(s) \log(\delta^{-1})$ by incorporating the~$\ell_\infty$-norm constraint into the estimation process.
One can envision applications where this improvement is crucial, e.g., in the context of generalized linear models~\cite{ostrovskii2021finite,marteau2019beyond}, where one has to take the union bound over an exponential in~$s$ number of events; this would render the~$s \log(\delta^{-1})$ term into~$s^2$ even in the fixed-probability regime, thus breaking the~$\wt O(s/n)$ parametric rate; see~\cite[Sec.~4]{ostrovskii2021finite} for a more technical discussion. 
This improvement comes at the very modest price of~$\log^2(s)$ inflation of the~$\delta$-independent term. 

Finally, it is an interesting question whether the~$\log^2(s)$ factor in~\eqref{eq:l2-core} is an artifact of our analysis.
This factor appears when we bound the measurements of random processes~$\langle \xi, \vphi * \xi \rangle_n$ and~$\|\vphi * \xi \|_{n,2}^2$ at extremal filters~$\vphi$, i.e.~such that~$\F_{n}[\vphi]$ is an~$s$-sparse vector with nonzero entries of the same magnitude. To this end, we use the classical oversampling trick, extracting ordinary convolution~$(\xi * \vphi)_t,$~$t \in [n]_{\pm}$ from the cyclic convolution of the noise, modulo~$4n+1$, with the filter~$\vphi$ padded with zeroes on the sides; see e.g.~\eqref{eq:Q-oper-prelim}. Then,~$\log(s)$ appears as the~$\ell_\infty$-norm of an oversampled filter with a sparse DFT, i.e. the sum of the Dirichlet kernel over a subset of~$s$ nodes of the DFT grid (see Lemma~\ref{lem:oversampling-linf-sparse}). It is unclear how to sidestep the Dirichlet kernel here.




\subsection{One-sided estimation}
\label{sec:l2-causal}

%

Proceeding as in Section~\ref{sec:l2-core} but replacing Proposition~\ref{prop:hybrid-oracle} with Proposition~\ref{prop:hybrid-oracle-causal}, we arrive at the estimator of the form~$\wh\vphi_+ * y$, where~$\wh\vphi_+$ is an optimal solution to the optimization problem
\begin{equation}
\label{opt:l2con-half}
\min_{\vphi \in \C_{2n}^+(\Z)}
\left\{ 
\|\Delta^{-n}[\vphi*y - y]\|_{n,2}^2:
\max\{\|\F_{2n}^+[\vphi]\|_{1},  c_1 s^2 \log(en) \|\F_{2n}^+[\vphi]\|_{\infty} \} \le \frac{c_1 s^2 \log(en)}{\sqrt{2n+1}}
\right\},
\end{equation}
where~$c_1$ is the constant from Proposition~\ref{prop:hybrid-oracle-causal} and~$\Delta$ is the lag operator:~$(\Delta x)_t = x_{t-1}$.
Note that the optimization problem~\eqref{opt:l2con-half} depends on the observations on~$[-2n, 2n]$ as before; yet, now we minimize the~$\ell_2$ norm of the residual on the ``right half'' interval~$[0,2n]$ rather than on the ``core''~$[-n,n]$.
Respectively, Theorem~\ref{th:l2con-core} can be adapted to the new estimator, which results in a guarantee for the MSE on the ``right-hand'' subdomain~$[0,2n]$, as follows.
\begin{theorem}
\label{th:l2con-half}
In the assumptions of Proposition~\ref{prop:hybrid-oracle-causal}, an optimal solution~$\wh\vphi_+ \in \C_{2n}^+(\Z)$ to~\eqref{opt:l2con-half} satisfies, for arbitrary~$\delta \in (0,1)$, the following inequality with probability at least $1-\delta$: 
\begin{equation}
\label{eq:l2-half}
\left\|\Delta^{-n}[\wh\vphi^+ * y - x]\right\|_{n,2}^2 
\lasymp \sigma^2 \left(s^2\log^2(en) + \log(e\delta^{-1}) \right) \log^2(es).
\end{equation}
\end{theorem}
The proof of Theorem~\ref{th:l2con-half} is carried out by adapting the proof of Theorem~\ref{th:l2con-core} in a straightforward fashion, so we omit it here.
The key difference of this bound with the bound~\eqref{eq:l2-core} of Theorem~\ref{th:l2con-half} is that instead of the term~$s \log n \asymp \log \left( {n \choose s} \right)$, in~\eqref{eq:l2-half} we have~$s^2 \log^2(n) \asymp \log\big( {n \choose s^2 \log n} \big)$, essentially because of the deteriorated bound on the~$\ell_1$-norm in Proposition~\ref{prop:causal-bound} -- which, in turn, leads to the increased number of extreme points of the corresponding~$\ell_1/\ell_\infty$-polytopal ball.
%
Motivated by the near-optimality of Proposition~\ref{prop:hybrid-oracle-causal}, see Appendix~\ref{apx:one-sided}, we conjecture that Theorem~\ref{th:l2con-half} cannot be improved by more than a logarithmic factor, 
as stated below.
\begin{conjecture}
\label{conj:causal-optimal}
For some~$c_0, c, r > 0$, the following holds: for any~$s,n \in \N: n \ge c_0 s^2 \log^r(n)$, there is an~SIS~$X(w_1, ..., w_s)$ with~$w_1, ..., w_s \in \T$, such that, for any estimator~$\wh x = \wh x(y_{-3n}, ..., y_{0})$,
\[
\Risk_{n,\delta}(\wh x, X)
\ge 
c\sigma^2 \left(s^2\log^r(n) + \log(\delta^{-1}) \right).
\]
\end{conjecture}
\begin{remark}
\label{rem:prediction}
{\em Replacing the constraint~$\vphi \in \C_{2n}^+(\Z)$ in~\eqref{eq:l2-half} with~$\Delta^{-h} \vphi \in \C_{2n}^+(\Z)$ for~$h \in \N$, one can generalize estimator~\eqref{opt:l2con-half} to the ``prediction'' setup, where one must estimate~$x_t$ from noisy observations on a segment {\em not containing~$t$}, with an appropriate modification of Theorem~\ref{th:l2con-half} remaining valid.
Such a modification is straightforward and left to the reader (cf.~\cite[Sec.~3.2]{harchaoui2015adaptive}).}
\end{remark}
Our next goal is to estimate~$x_t$ on the full domain~$[-2n,2n]$.
Of course, assuming~$x \in \XX(\T^s)$ instead of~$x \in \XX(\D^s)$, we can simply run the estimation procedure~\eqref{opt:l2con-half} on the observations~$y_t$ and their ``time-reversed'' version~$\wt y_{t} := y_{-t}$. 
Yet, such a naive estimate would suffer from the mean-squared error of order~$s^2$, cf.~\eqref{eq:l2-half}, falling short of matching the lower bound in~\eqref{eq:intro-lower}. 
Fortunately, there is a better approach: it turns out that a nearly linear in~$s$ bound can be ensured by employing instead a multiscale technique that we are about to present next.


\subsection{Full-domain estimation and minimax risk on~$\XX_s$}
\label{sec:l2-full}

Our technique for constructing the full-domain estimate can be summarized as follows.
Starting from the initial interval~$[-n,n]$, take the next interval adjacent to the previous one on the right (resp.~left) and shorter by a constant factor; fit an estimate on this interval; repeat this process until the interval length shrinks to~$O(s)$.
Formally, assuming first that~$s$ and~$n [\ge \odima{3s}]$ are powers of~$3$, we implement the following procedure:
\begin{itemize}
\item[1.] Let~$\wh x_t^\full := (\wh\vphi^{(0)} * y)_t^{\vphantom{\full}}$ for~$|t| \le n$,  where~$\wh\vphi^{(0)}$ is an optimal solution to~\eqref{opt:l2con-core}. 
\item[2.] Letting~$n_k := n3^{-k}$ and~$h_k := 2n-2n_k = 2n(1-3^{-k})$, define two families of intervals:
\begin{equation}
\label{eq:zeno-grid}
\begin{aligned}
I^{(k,+)} := \left(2n - n_{k-1}, 2n - n_{k}\right] 
	&= (h_k-n_k, h_k + n_k],\notag\\
I^{(k,-)} := \left[n_{k} - 2n, n_{k-1} - 2n\right) 
	&= [-h_k-n_k, -h_k+n_k)  \notag
	\end{aligned}
\;\; \text{for}\;\; k \in \left\{1, ..., K := \log_3\left(\frac{n}{\odima{3s}}\right)\right\}.
\end{equation}
Clearly, we have~$n_K = \odima{3s}$ and~$|I^{(k,+)} \cap \Z| = |I^{(k,-)} \cap \Z| = 2n_k \ge \odima{6s}$.
Moreover, it holds that~$\bigsqcup_{k = 1}^K I^{(k,+)} = (n,2n-\odima{3s}]$ and~$\bigsqcup_{k = 1}^K I^{(k,-)} = [-2n+\odima{3s},-n)$.
In other words,~$I^{(k+1,+)}$ (resp.~$I^{(k+1,-)}$) is appended to the previous interval~$I^{(k,+)}$ (resp.~$I^{(k,-)}$) on the right (resp.~left), taking up~$2/3$ of the remaining gap~$[2n-n_k,2n]$ (resp.~$[-2n,-2n+n_k]$); this process is repeated until the interval length becomes~$O(s)$, at which point we terminate.
\item[3.] For~$k \in \{1, ..., K\}$, we solve the counterparts of~\eqref{opt:l2con-core} on~$I^{(k,+)}$ and~$I^{(k,-)}$ instead of~$[-n,n]$:
\begin{equation}
\label{eq:side-filters}
\bald\nvsps
\wh\vphi^{(k,+)} &\in \Argmin_{\vphi \in \C_{n_k}(\Z)}
\left\{ \|\Delta^{-h_k}[\vphi*y - y]\|_{n_k,2}:\, \max\{\|\F_{n_k}[\vphi]\|_{1},  \odima{2s}\|\F_{n_k}[\vphi]\|_{\infty} \} \le \frac{\odima{2s}}{\sqrt{2n_k+1}} \right\},\\
\wh\vphi^{(k,-)} &\in \Argmin_{\vphi \in \C_{n_k}(\Z)}
\left\{ \|\Delta^{+h_k}[\vphi*y - y]\|_{n_k,2}:\, \max\{\|\F_{n_k}[\vphi]\|_{1},  \odima{2s}\|\F_{n_k}[\vphi]\|_{\infty} \} \le \frac{\odima{2s}}{\sqrt{2n_k+1}} \right\}.
\eald
\end{equation}
\item[4.]
We let~$\wh x_t^\full := (\wh\vphi^{(k,\pm)} * y)_t^{\vphantom{\full}}$ for~$t \in I^{(k,\pm)}$ respectively, and~$\wh x_t^\full := y_t^{\vphantom{\full}}$ for~$t : |t| > 2n-\odima{3s}$. 
\end{itemize}
Note that, the resulting estimate~$(\wh x_t^\full, |t| \le 2n)$ uses the observations~$(y_{t}^{\vphantom\full}, |t| \le 2n)$, as desired.
The following result---essentially, a corollary of Theorem~\ref{th:l2con-core}---shows that our technique leads to essentially the same error on the whole domain as ``in the core,'' up to an~$O(\log(n/s))$ overhead.
\begin{corollary}
\label{cor:l2-full}
Assume~$x \in \XX_s(\C)$,~$n \ge \odima{3s}$, and~$n,s$ are powers of~$3$. 
Estimate~$\wh x_t^\full$ satisfies
\[
\big\|\wh x^\full - x\big\|_{2n,2}^2 \le 
\odima{80}\, \sigma^2 \left( \odima{s}\log(2n+1) + \odima{3s} + \log(\delta^{-1})  \right) \log^2(e^4 s) \,\log_3(n/s).
\]
\end{corollary}

Finally, as we discuss next, the assumption that~$n,s$ are powers of 3 can be easily removed.
\begin{remark}
\label{rem:nontriadic}
{\em In the general case, the procedure described above can be adjusted as follows.
Let~$s_0 := 3^{\lceil \log_3(s) \rceil} \in [s,3s]$,~$n_0 := 3^{\lfloor \log_3(n) \rfloor}$, and~$K_0 := \log_3(n_0/s_0) - 1$.
Now, cover~$[-2n,2n]$ with at most~$3$ overlapping intervals of halfwidth~$2n_0$; specifically,~consider the decomposition
\[
[-2n,2n] \;=\;  [-2n, -2n+4n_0] \,\cup\, [-2n_0, 2n_0] \,\cup\, [2n-4n_0, 2n].
\]
When~$n = 3^m$ for~$m \in \N$, one has~$n_0 = n$, and all three intervals~$\wt I^{(-1)}, \wt I^{(0)}, \wt I^{(1)}$ in the right-hand side coincide. 
When~$n = 3^m-1$, the intervals have almost no overlap (this is ``the worst case'').
Now, let us run the procedure separately on each of the three intervals (with~$s_0, n_0, K_0$ instead of~$s,n,K$) to form the estimates~$\wt x_t{}^{(-1)}, \wt x_t{}^{(0)}, \wt x_t{}^{(1)}$, each defined on the respective interval. 
Output~$\wh x_t = \wt x_t{}^{(0)}$ for~$t \in \wt I^{(0)} \setminus (\wt I^{(-1)} \cup \wt I^{(1)})$,~$\wh x_t = \wt x_t{}^{(\pm 1)}$ for~$t \in \wt I^{(\pm 1)} \setminus (\wt I^{(0)} \cup \wt I^{(\mp 1)})$, and
in the overlaps simply average the estimates:~$\wt x_t = \half(\wt x_t{}^{(0)} + \wt x_t{}^{(\pm 1)})$ for~$t \in \wt I^{(0)} \cap \wt I^{(\pm 1)}$. 
By the triangle inequality and union bound, the guarantee of Corollary~\ref{cor:l2-full} is preserved up to a constant factor.}
\end{remark}

Note that the result we have just established immediately implies the previously announced bound~\eqref{eq:l2-core-intro-risk-minimax} for the minimax~$\delta$-risk over~$\XX_s$.


\subsection{Signal detection}
\label{sec:detection}

%
We now focus on the {detection} counterpart of the estimation problem considered previously. 
Namely, we are to test the null hypothesis that the observations~$(y_t, |t| \le 2n)$ contain no signal,
\begin{equation}
\label{def:hyp-nul}
\Hyp_{0,n} := \{x \in \C(\Z): \| x \|_{2n,2} = 0\},
\end{equation}
against the alternative that~$x$ belongs to some SIS of dimension~$s$ and is of large enough norm:
\begin{equation}
\label{def:hyp-alt}
\Hyp_{s,n}(\sL) := \{x \in \XX_s(\C): \|x\|_{2n,2} \ge \sL\}.
\end{equation}
Note that this setup naturally extends that of detecting a sparse signal (see e.g.~\cite{ingster2010detection}): the latter one can be recovered if we replace~$\XX_s$ in~\eqref{def:hyp-alt} with~$\smash{\XX(\T_{2n}^{(s)})}$. 
We do not consider here the superficially more general setup of~\cite{jn-2013}, where the null hypothesis allows~$x$ to come from a known SIS: our methodology can be extended to accommodate for this, but at the cost of obscuring the essential ideas in the proofs.
Our testing procedure compares the statistic
\begin{equation}
\label{eq:detection-statistic}
\wh T_n(y) := \| y \|_{2n,2}^2 - \|y - \wh x^\full\|_{2n,2}^2,
\end{equation}
where~$\wh x^\full$ is the full-domain estimate from Section~\ref{sec:l2-full}, against the threshold value~$\sL_0^2$ that corresponds to the in-probability bound for~$\|\wh x^\full - x \|_{2n,2}^2$ established in Corollary~\ref{cor:l2-full}, and rejects the null whenever the threshold is surpassed. 
It admits the following statistical guarantee.

\begin{theorem}
\label{th:detection}
Consider testing~$\Hyp_{0,n}$, cf.~\eqref{def:hyp-nul}, against~$\Hyp_{s,n}(\sL)$, cf.~\eqref{def:hyp-alt}, assuming~$n \ge \odima{3s}$ and
\[
\sL^2 \ge \sL_0^2 := \odima{300}\,\sigma^2 \big(\, 2s\log (2n+1) + 6s + \log(e^3s) \log\left(6{\log_3(n/s)}{\delta^{-1}}\right) \big) \log(e^4 s) \, \log_3(n/s).
\]
The test rejecting~$\Hyp_{0,n}$ when~$\wh T_n(y) > \odima{\tfrac{3}{8}} \sL_0^2$ makes errors of either type with probability~at most~$\delta$.
\end{theorem}
%
\begin{remark}
{\em The proof of Theorem~\ref{th:detection}, given in the next section, is similar to the combined proof of Theorem~\ref{th:l2con-core} and~Corollary~\ref{cor:l2-full}, because the various error terms that must be handled in the proof of Theorem~\ref{th:l2con-core} also appear in the decomposition of the test statistic. More precisely, the Type I error occurs due to the pure-noise convolution term~$\langle \xi, \wh \vphi * \xi \rangle_n$, while the remaining (signal-dependent) terms contribute to the Type II error. By the same arguments as those proving Theorem~\ref{th:l2con-core} and~Corollary~\ref{cor:l2-full}, the test statistic is bounded as~$O(\sL_0^2)$ under~$\Hyp_{0,n}$ and~$O(\sL_0^2 + \sL^2)$ under~$\Hyp_{s,n}(\sL)$; this allows to distinguish between the two hypotheses when~$\sL \gasymp \sL_0$.}
\end{remark}

\begin{remark}
{\em The detection boundary~$\sL = \sL_0$ established in Theorem~\ref{th:detection} is known to be near-optimal, up to a logarithmic factor, for sparse regression in the asymptotic highly sparse regime, namely when~$s,n \to \infty$ with~$s \asymp d^{\alpha}$ for~$\alpha \in (0, 1/2)$, where~$d$ is the ambient dimension~\cite{ingster2010detection}; this translates to~$s \sim n^{\alpha}$, in the denoising version of our problem, i.e.~when the condition~$x \in \XX_s$ in the definition of~$\Hyp_{s,n}(\sL)$, cf.~\eqref{def:hyp-alt}, is replaced with~$x \in \XX(\T_{2n}^{(s)})$. 
On the other hand, in the ``moderately sparse'' regime ($\half \le \alpha < 1$) the asymptotic detection boundary is known to be~$\sL^2 \asymp s^{1/2} (\asymp \sL_0)$, cf.~\cite{ingster2010detection}, i.e.~detection is possible at low signal-to-noise ratios. We anticipate that Theorem~\ref{th:detection} can also be extended in this direction, but we have not attempted to do this.}
\end{remark}

\section{Proofs of statistical results}
\label{sec:stat-proofs}

\paragraph{Preliminaries for proving Theorems~\ref{th:l2con-core} and~\ref{th:detection}.}
Let~$\wh\vphi = \wh\vphi^{(0)} \in \C_{n}(\Z)$ be optimal in~\eqref{opt:l2con-core}, and let~$\wh\vphi^{(k,\pm)} \in \C_{n_k}(\Z)$, for~$k \in \{1,.., K\}$ and~$n_k = n3^{-k}$, be as in~\eqref{eq:side-filters}; in other words,
\[
\wh x_t^\full = (\wh \vphi^{(k,\pm)} * y)_t^{\vphantom\full} \;\;\text{for}\;\; t \in I^{(k,\pm)}.
\]
Finally, let~$\vphi^{(0)} \in \C_{n}(\Z)$ and~$\vphi^{(k,\pm)} \in \C_{n_k}(\Z)$, for~$k \in \{1,.., K\}$, be feasible solutions to~\eqref{opt:l2con-core} and~\eqref{eq:side-filters} respectively, as exhibited in Proposition~\ref{prop:hybrid-oracle}. 
That is, letting~$X$ be the shift-invariant subspace containing~$x$ with~$\dim(X) = s$,~$\vphi^{(k,\pm)}$ are reproducing for~$X$ and satisfy 
\begin{equation}
\label{eq:side-filter-norms}
\|\F_{n_k} [\vphi^{(k,\pm)}]\|_1 \le \frac{\odima{2s}}{\sqrt{2n_k+1}}, \quad
\|\vphi^{(k,\pm)}\|_2 \le \frac{\odima{\sqrt{2s}}}{\sqrt{2n_k+1}}, \quad
\|\F_{n_k} [\vphi^{(k,\pm)}]\|_\infty \le \frac{\odima{1}}{\sqrt{2n_k+1}};
\end{equation}
similarly for~$\vphi^{(0)} = \vphi^\star$ with~$n_0 = n$. 
We find it somewhat more convenient to prove Theorem~\ref{th:detection} first, and subsequently recycle the estimates of various error terms in the proof of Theorem~\ref{th:l2con-core}.

\subsection{Proof of Theorem~\ref{th:detection}}
\textbf{\em Type II error.}
\proofstep{1}.
Consider the following decomposition of the test statistic:
\begin{equation}
\label{eq:detection-II-split}
\begin{aligned}
\wh T_n(y) 
&= \|y\|_{2n,2}^2 - \|y - \wh x^\full\|_{2n-3s,2}^2 \\
&= \|y\|_{2n-3s,2}^2 - \|y - \wh x^\full\|_{2n-3s,2}^2 + \|y\|_{2n,2}^2 - \|y\|_{2n-3s,2}^2 \\
&= \underbrace{\|y\|_{n,2}^2 - \big\|y - \wh \vphi^{(0)} * y \big\|_{n,2}^2}_{\wh T_n^{(0)}(y)}
	+ \sum_{k = 1}^K \underbrace{\|\Delta^{\mp h_k}[y]\|_{n_k,2}^2 - \big\|\Delta^{\mp h_k}[y - \wh \vphi^{(k,\pm)} * y] \big\|_{n_k,2}^2}_{\wh T_n^{(k,\pm)}(y)} \\ 
&\hphantom{= \|y\|_{2n-3s,2}^2 - \|y - \wh x^\full\|_{2n-3s,2}^2\;}
	+ \|y\|_{2n,2}^2 - \|y\|_{2n-3s,2}^2.
\end{aligned}
\end{equation}
Here~$\wh T_n^{(k,+)}$ or~$\wh T_n^{(k,-)}$ examines the respective~$I^{(k,+)}$ or~$I^{(k,-)}$, idem for~$\wh T_n^{(0)}$ and~$I^{(0)} := [-n,n]$; the term~$\|y\|_{2n,2}^2 - \|y\|_{2n-3s,2}^2$ appears as~$\wh x_t^\full = y_t^{\vphantom{\full}}$ for~$t: |t| > 2n-3s$.
Focusing on~$I^{(0)}$ w.l.o.g.,
\begin{align}
\wh T_n^{(0)}(y) 
\ge T_n^{(0)}(y) 
&:=  \|y\|_{n,2}^2  - \|y - \vphi^{(0)} * y \|_{n,2}^2  \notag \\
&\;= \|y\|_{n,2}^2 - \sigma^2 \|\xi - \vphi^{(0)} * \xi \|_{n,2}^2 \notag \\
&\;= \|x\|_{n,2}^2 - \sigma^2 \|\vphi^{(0)} * \xi\|_{n,2}^2  + 2 \sigma^2 \Re \langle \xi, \vphi^{(0)} * \xi \rangle_n +  2\sigma \Re \langle \xi, x \rangle_{n}
\label{eq:detection-II-decomp}
\end{align}
where we first used that~$\vphi^{(0)}$ is feasible and reproducing for~$X$, and then expanded the squares.

\noindent\proofstep{2}.
Define the matrix-valued map~$\Filt_{p,n}: \C(\Z) \to \C^{(2p+1) \times (2p+2n+1)}$ encoding the convolution,
\begin{align}
\label{eq:convo-matrix}
\Filt_{p,n}(\vphi) 
&:= 
\begin{pmatrix}
\vphi_{-n}\; & \cdots & \cdots &  \cdots & \vphi_{n} & 0 & \cdots & \cdots & \cdots & 0\\
0 & \vphi_{-n} & \cdots & \cdots & \cdots & \vphi_{n} & 0 & \cdots & \cdots & 0\\
\vdots & \ddots & \ddots & \hphantom{\cdots} & \hphantom{\cdots} & \hphantom{\cdots} & \ddots & \ddots & \hphantom{\cdots} & \vdots \\
\vdots & \hphantom{\cdots} & \ddots & \ddots & \hphantom{\cdots} & \hphantom{\cdots}  & \hphantom{\cdots} & \ddots & \ddots & \vdots \\
0 & \cdots & \cdots & 0 & \vphi_{-n} & \cdots & \cdots & \cdots & \vphi_{n} & 0\\
0 & \cdots & \cdots & \cdots & 0 & \vphi_{-n} & \cdots & \cdots & \cdots & \vphi_{n}
\end{pmatrix},
\end{align}
and its counterpart~$\Circ_{p,n}: \C(\Z) \to \C^{(2p+2n+1) \times (2p+2n+1)}$ encoding the circular convolution,
\begin{align}
\label{eq:circ-convo-matrix}
\Circ_{p,n}(\vphi) 
&:= 
\begin{pmatrix}
\begin{matrix}
\vphi_{0} & \cdots & \cdots & \vphi_{n} & 0 & \cdots & 0 & \vphi_{-n} & \cdots & \vphi_{-1}\\
\vdots & \ddots & \hphantom{\cdots} & \hphantom{\cdots} & \ddots & \ddots & \hphantom{\cdots} & \ddots & \ddots & \vdots\\
\vphi_{-n+1} & \cdots & \vphi_{0} & \cdots & \cdots & \vphi_{n} & 0 & \cdots & 0 & \vphi_{-n}\\
\end{matrix}\\
\hline
\\
\Filt_{p,n}(\vphi)\\
\\
\hline
\begin{matrix}
\quad\vphi_{n} & 0 & \cdots & 0 & \vphi_{-n} & \cdots & \cdots & \vphi_0 & \cdots  & \vphi_{n-1}\\
\quad\vdots & \ddots & \ddots & \hphantom{\cdots} & \ddots & \ddots & \hphantom{\cdots} & \hphantom{\cdots} & \ddots & \vdots\\
\quad\vphi_{1} & \cdots & \vphi_{n} & 0 & \cdots & 0 & \vphi_{-n} & \cdots & \cdots & \vphi_{0}\\
\end{matrix}\\
\end{pmatrix}.
\end{align}
Also, define the ``slicing'' map~$\{\cdot\}_n: \C(\Z) \to \C^{2n+1}$ by~$\{x\}_n := ( x_{-n}, \; \cdots \;  x_{n} )^\top$. 
For~$\vphi \in \C_n(\Z)$,
\begin{equation}
\label{eq:conv-matrix-form}
\{\vphi * \xi\}_{n} = \Filt_{n,n}(\vphi) \, \{ \xi \}_{2n},
\end{equation}
whence for the negative term in~\eqref{eq:detection-II-decomp} we get
\begin{align}
\label{eq:Q-matrix}
\big\|\vphi^{(0)} * \xi \big\|_{n,2}^2 
	&= \{\xi\}_{2n}^\htop Q^{(0)} \{\xi\}_{2n}^{\vphantom\htop}
\quad\text{with}\quad
Q^{(0)} 
	:= \Filt_{n,n}(\vphi^{(0)})^\htop  \Filt_{n,n}(\vphi^{(0)}).
\end{align}
Here the matrix~$Q^{(0)}$ is positive-semidefinite, and we can explicitly compute its trace:
\begin{equation}
\label{eq:Q-trace}
\tr (Q^{(0)} ) = \|\Filt_{n,n}(\vphi^{(0)})\|_{\sF}^2 
= (2n+1)\|\vphi^{(0)}\|_2^2 
\stackrel{\eqref{eq:side-filter-norms}}{\le} \odima{2s}.
\end{equation}
By the DFT diagonalization property,~$\Circ_{p,n}(\vphi)$ with~$\vphi \in \C_n(\Z)$ is unitary equivalent to the diagonal matrix~$\sqrt{2p+2n+1}\, \Diag(\F_{p+n}[\vphi])$. 
As~$\Filt_{n,n}(\vphi)$ is a submatrix of~$\Circ_{n,n}(\vphi)$, we get
\begin{align}
\label{eq:Q-oper-prelim}
\big\| Q^{(0)} \big\|_{\op} 
\le \| \Filt_{n,n}(\vphi^{(0)}) \|_{\op}^2
\le \| \Circ_{n,n}(\vphi^{(0)}) \|_{\op}^2
= (4n+1) \|\F_{2n}[\vphi^{(0)}]\|_{\infty}^2.
\end{align}
By~\eqref{eq:side-filter-norms},~$\sqrt{2n+1} \F_n[\vphi^{(0)}]$ belongs to the set~$\Comp^{2n+1}(\odima{2s})$, see Appendix \ref{apx:extreme-points}, for which
\[
\max_{w \in \Comp^{2n+1}(\odima{2s})} \| \F_{2n}^{\vphantom\pinv}[\F_n^{\,-1}[w]] \|_\infty
\le \sqrt{\frac{2n+1}{4n+1}} (\log\odima{s} + 3)
= \sqrt{\frac{2n+1}{4n+1}} \log(e^3 \odima{s}).
\]
When combined with~\eqref{eq:Q-oper-prelim} this gives
\vspace{-0.1cm}\begin{align}
\label{eq:Q-oper}
\big\| Q^{(0)} \big\|_{\op} 
\le \log^2(e^3\odima{s}),
\end{align}
whence for the Frobenius norm (cf.~\eqref{eq:Q-trace}):
\begin{equation}
\label{eq:Q-frob}
\big\| Q^{(0)} \big\|_{\sF}
\le \sqrt{\big\|Q^{(0)} \big \|_{\op} \tr \big(Q^{(0)} \big)}
\le \sqrt{2s}\log(e^3\odima{s}).
\end{equation}
Now, recall the deviation bound for an indefinite Hermitian form of a complex Gaussian vector,
\begin{equation}
\label{eq:quad-form-indef}
\Prob \big\{ \half \xi^\htop Q \xi \le \tr(Q) + \| Q \|_{\sF} \sqrt{2 \log(\delta^{-1})} + \|Q\|_{\op} \log(\delta^{-1}) \big\} \ge 1-\delta.
\end{equation}
(It can be easily verified by combining the right and left tail bounds for~$\chi_2^2$ from~\cite[Lem.~1]{lama2000}.)
Plugging~\eqref{eq:Q-matrix}--\eqref{eq:Q-trace} and~\eqref{eq:Q-oper}--\eqref{eq:Q-frob} into~\eqref{eq:quad-form-indef} we conclude that with probability at least~$1-\delta$,
\begin{align}
-\sigma^2 \big\|\vphi^{(0)} * \xi \big\|_{n,2}^2 
&\ge -2\sigma^2 \big(\odima{2s} + 2\log(e^3\odima{s})\sqrt{s \log(\delta^{-1})} + \log^2(e^3\odima{s})\log(\delta^{-1}) \big) \notag\\
&\ge -2\sigma^2 \left(\odima{3s} + \odima{2}\log^2(e^3 \odima{s})\log(\delta^{-1}) \right).
\label{eq:detection-II-term1}
\end{align}

\noindent\proofstep{3}. We can now proceed similarly to estimate the next term in~\eqref{eq:detection-II-decomp}.
To this end, we note that
\begin{align}
\label{eq:S-matrix}
2\Re \langle \xi, \vphi^{(0)} * \xi \rangle_n 
	&= \{\xi\}_{2n}^\htop S^{(0)} \{\xi\}_{2n}^{\vphantom\htop}
\,\quad\text{with}\quad
S^{(0)} 
	:= 
\begin{bmatrix}
\mathsf{O}_{n,4n+1}\\
\Filt_{n,n}(\vphi^{(0)})\\
\mathsf{O}_{n,4n+1}
\end{bmatrix}
+ 
\begin{bmatrix}
\mathsf{O}_{n,4n+1}\\
\Filt_{n,n}(\vphi^{(0)})\\
\mathsf{O}_{n,4n+1}
\end{bmatrix}^\htop
\end{align}
where~$\Filt_{n,n}(\vphi)$ is given by~\eqref{eq:convo-matrix} and~$\mathsf{O}_{a,b}$ is the~$a \times b$ zero matrix. 
Note that~$S^{(0)}$ is Hermitian, but not necessarily positive-semidefinite. We first estimate its Frobenius and operator norms:
\vspace{-0.1cm}\begin{align}
\label{eq:S-frob}
\big\|S^{(0)}\big\|_{\sF}\,\,
\le 2 \|\Filt_{n,n}(\vphi^{(0)})\|_{\sF}\,\,
&\stackrel{\eqref{eq:Q-trace}}{\le} \odima{2\sqrt{2s}},\\
%
\big\|S^{(0)}\big\|_{\op}
\le 2 \| \Filt_{n,n}(\vphi^{(0)}) \|_{\op}
&\stackrel{\eqref{eq:Q-oper}}{\le} \odima{2}\log(e^3\odima{s}).
\end{align}
Meanwhile,~$\tr(S^{(0)}) = 2(2n+1)\Re[\vphi^{(0)}]_0$ and, letting~$\veps^{0} \in \C(\Z)$ be the discrete unit pulse,
\[
|\vphi_0| = |\langle \veps^0, \vphi \rangle_{n}| \le \|\F_{n}[\veps^0]\|_\infty \, \|\F_{n}[\vphi]\|_1 =  (2n+1)^{-1/2} \|\F_{n}[\vphi]\|_1 \quad \text{for any} \;\; \vphi \in \C_n(\Z).
\]
As the result,
\begin{equation}
\label{eq:S-trace}
|\tr(S^{(0)})| 
\le 
	2\sqrt{2n+1} \big\|\F_{n}\big[\vphi^{(0)}\big]\big\|_1
\stackrel{\eqref{eq:side-filter-norms}}{\le} 
\odima{4}s.
\end{equation}
Plugging~\eqref{eq:S-matrix}--\eqref{eq:S-trace} into~\eqref{eq:quad-form-indef} we conclude that with probability~$\ge 1-\delta$,
\begin{align}
2 \sigma^2 \Re \langle \xi, \vphi^{(0)} * \xi \rangle_n 
&\ge -2\sigma^2 \big(\odima{4s} + \odima{4}\sqrt{s\log(\delta^{-1})} + \odima{2}\log(e^3 s)\log(\delta^{-1}) \big) \notag\\
&\ge -\odima{4}\sigma^2 \big(\odima{3s} + \odima{2}\log(e^3 s)\log(\delta^{-1}) \big).
\label{eq:detection-II-term2}
\end{align}
Finally, observe that~$\Re \langle \xi, x \rangle_n \sim \cN(0,\|x\|_{n,2}^2)$ which implies that with probability at least~$ 1-\delta$,
\begin{align}
2\sigma \Re \langle \xi, x \rangle_n 
\ge -2\sigma \|x\|_{n,2}  \sqrt{2\log(\delta^{-1})}
\ge -\frac{1}{8}\|x\|_{n,2}^2 - 16\sigma^2 \log(\delta^{-1}).
\label{eq:detection-II-term3}
\end{align}
Plugging~\eqref{eq:detection-II-term1},~\eqref{eq:detection-II-term2}, and~\eqref{eq:detection-II-term3} into~\eqref{eq:detection-II-decomp}, we now conclude that, with probability at least~$1-\delta$,
\begin{equation}
\wh T_n^{(0)}(y) \ge \frac{7}{8} \|x\|_{n,2}^2 - \odima{2}\sigma^2 \big(\odima{9s} + \odima{14}\log^2(e^3 \odima{s})\log(3\delta^{-1})\big).
\label{eq:detection-II-core}
\end{equation}

\noindent\proofstep{4}.
Of course, the exact same bound applies, separately, for each~$\wh T_n^{(k,\pm)}$,~$k \in \{1, ..., K\}$.
Finally, applying the deviation bound~\eqref{eq:quad-form-indef} to~$\|\xi\|_{2n,2}^2 - \|\xi\|_{2n-3s,2}^2 \sim \chi^2_{12s}$, with probability at least~$1-\delta$
\begin{align}
\label{eq:detection-I-edge}
\|\xi\|_{2n,2}^2 - \|\xi\|_{2n-3s,2}^2 
&\le 12 s + 4\sqrt{3s\log(\delta^{-1})} + 2\log(\delta^{-1})
\le \odima{15 s} + \odima{6}\log(\delta^{-1}).
\end{align}
By Cauchy-Schwarz, with the same probability it holds that
\begin{align}
\label{eq:detection-II-edge}
\|y\|_{2n,2}^2 - \|y\|_{2n-3s,2}^2 
&\ge \frac{7}{8} \big( \|x\|_{2n,2}^2 - \|x\|_{2n-3s,2}^2 \big) 
		-7 \sigma^2 \big(15 s + 6\log(\delta^{-1}) \big).
\end{align}
Combining~\eqref{eq:detection-II-core} and~\eqref{eq:detection-II-edge} with~\eqref{eq:detection-II-split} and invoking the union bound, we get with prob.~$\ge 1-\delta$:
\[
\begin{aligned}
\wh T_n(y)  
&\ge 
	\frac{7}{8}\sL^2
-2(K+1) \sigma^2 \left(123s + 72\log^2(e^3 s) \log(6(K+1)\delta^{-1}) \right) \\
&\ge
   \frac{7}{8}\sL_0^2
-144(K+1) \sigma^2 \left(2s + \log^2(e^3 s) \log(6(K+1)\delta^{-1}) \right)
\ge
	\frac{7}{8}\sL_0^2 - \odima{\frac{1}{2}\sL_0^2}
= 	\odima{\frac{3}{8}\sL_0^2}.
\end{aligned}
\]

\textbf{\em Type I error.}
Under the null hypothesis and upon factoring out~$\sigma^2$, identity~\eqref{eq:detection-II-split} becomes
\begin{equation}
\label{eq:detection-I-split}
\begin{aligned}
\wh T_n^{\vphantom{(0)}}(\xi) 
= \wh T_n^{(0)}(\xi) + \left( \sum_{k = 1}^K \wh T_n^{(k,\pm)}(\xi) \right) + \|\xi\|_{2n,2}^2 - \|\xi\|_{2n-3s,2}^2.
\end{aligned}
\end{equation}
The term~$\|\xi\|_{2n,2}^2 - \|\xi\|_{2n-3s,2}^2 \sim \chi^2_{12s}$ has already been bounded in~\eqref{eq:detection-I-edge}. 
As for the first term,
\begin{align}
\wh T_n^{(0)}(\xi)
= \|\xi\|_{n,2}^2 - \|\xi - \wh\vphi^{(0)} * \xi \|_{n,2}^2 
\le 2 \Re \langle \xi, \wh\vphi^{(0)} * \xi \rangle_n.
\label{eq:detection-II-term}
\end{align}
Note that since~$\wh \vphi^{(0)}$ is random, we cannot control the right-hand side by recycling the argument from~(\proofstep{3}). 
Instead, we are going to bound the supremum of the random process indexed on the feasible set~$\Phi$ of~\eqref{opt:l2con-core}. 
As shown in Proposition~\ref{prop:cap-comp}, all extreme points of~$\Phi$ are contained in
\begin{equation}
\label{eq:extreme-points-l1-linf}
\Ext_{s,n} := \left\{ \frac{\odima{1}}{\sqrt{2n+1}} \F_n^{\,-1} \left[\sum_{k = 1}^{\odima{2s}} z_k e^{j_k} \right]: \; \{ j_1, ..., j_{\odima{2s}} \} \subset \{0,1,...,2n\}, \; (z_1, ..., z_{\odima{2s}}) \in \T^{\odima{2s}} \right\}
\end{equation}
where~$e^0, ..., e^{2n}$ is the canonical basis of~$\R^{2n+1}$; that is, the DFT~$w = \F_n[\vphi]$ of an extremal~$\vphi$ is an~$\odima{2s}$-sparse vector with nonzero entries of constant norm.
Now, let~$\cJ := \{j_1, ..., j_{\odima{2s}}\}$ and~$\vec z := (z_1, ..., z_{\odima{2s}})$, and let~$\vphi_{\cJ, \vec z}$ be the corresponding element of~$\Ext_{s,n}.$
Then, following~\eqref{eq:S-matrix}, 
\begin{align*}
2\Re \langle \xi, \vphi_{\cJ, \vec z} * \xi \rangle_n 
	&= \{\xi\}_{2n}^\htop S_{\cJ, \vec z} \{\xi\}_{2n}^{\vphantom\htop}
\,\quad\text{with}\quad
S_{\cJ, \vec z}
	:= 
\begin{bmatrix}
\mathsf{O}_{n,4n+1}\\
\Filt_{n,n}(\vphi_{\cJ, \vec z})\\
\mathsf{O}_{n,4n+1}
\end{bmatrix}
+ 
\begin{bmatrix}
\mathsf{O}_{n,4n+1}\\
\Filt_{n,n}(\vphi_{\cJ, \vec z})\\
\mathsf{O}_{n,4n+1}
\end{bmatrix}^\htop.
\end{align*}
Moreover, we can bound~$\|S_{\cJ,\vec z}\|_{\sF}$,~$\|S_{\cJ,\vec z}\|_{\op}$, and~$\tr(S_{\cJ,\vec z})$ by proceeding as in~\eqref{eq:S-frob}--\eqref{eq:S-trace}; this gives
\[
\begin{aligned}
\big\|S_{\cJ,\vec z}\big\|_{\sF}\,\,
&\le 2 \|\Filt_{n,n}(\vphi_{\cJ,\vec z})\|_{\sF}
= \odima{2\sqrt{2s}},\\
\big\|S_{\cJ,\vec z}\big\|_{\op}
&\le 2 \| \Filt_{n,n}(\vphi_{\cJ,\vec z}) \|_{\op}
\le \odima{2\log(e^3 s)},\\
|\tr(S_{\cJ,\vec z})| 
&\le 
	2\sqrt{2n+1} \|\F_{n}[\vphi_{\cJ,\vec z}]\|_1
=	\odima{4s}.
\end{aligned}
\]
As the result, any (fixed)~$\vphi_{\cJ, \vec z}$ with prob.~$\ge 1-\delta$ satisfes the same inequality as in~\eqref{eq:detection-II-term2}, namely
\begin{equation}
\label{eq:detection-II-term2-one-corner}
2 \Re \langle \xi, \vphi_{\cJ, \vec z} * \xi \rangle_n 
\le 4 \big(3s + 2\log(e^3 s)\log(\delta^{-1}) \big).
\end{equation}
Now, in order to get supremum over~$\vphi_{\cJ, \vec z} \in \Ext_{s,n}$ let us first rid of the phase vector~$\vec z$ in~$\vphi_{\cJ,\vec z}$. 
To this end, note that for any~$z \in \T$ and~$0 \le e^{j} \le 2n$, 
\begin{align*}
\Re \langle \xi, \F_n^{\,-1}[z e^{j}] * \xi \rangle_n 
&= \Re (z\langle \xi, \F_n^{\,-1}[e^{j}] * \xi \rangle_n ) \notag\\
&\le | \langle \xi, \F_n^{\,-1}[e^{j}] * \xi \rangle_n | \notag\\
&\le \sqrt{2} \max \big\{ | \Re \langle \xi, \F_n^{\,-1}[e^{j}] * \xi \rangle_n |,\; | \Im \langle \xi, \F_n^{\,-1}[e^{j}] * \xi \rangle_n | \big\} \notag\\
&= \sqrt{2} \max_{w \in \{\pm 1,\pm i \}}
				\Re \langle \xi, \F_n^{\,-1}[w e^{j}] * \xi \rangle_n.
\end{align*}
Whence for any~$\vec z \in \T^{\odima{2s}}$ and~$\cJ \subset \{0,...,2n\}$ with~$|\cJ| = \odima{2s}$, 
\begin{align}
\Re \langle \xi, \vphi_{\cJ, \vec z} * \xi \rangle_n 
&\le \sqrt{2} \sum_{k = 1}^{\odima{2s}} \max_{w_k \in \left\{\pm 1, \pm i\right\}}
	\Re \langle \xi, \F_n^{\,-1}[w_k e^{j_k}] * \xi \rangle_n 
=
	\sqrt{2} \max_{\vec w \in \left\{\pm 1, \pm i\right\}^{\odima{2s}}}
				\Re \langle \xi, \vphi_{\cJ,\vec w} * \xi \rangle_n.				
\label{eq:phase-matching}
\end{align}
The total number of~$\vphi_{\cJ,\vec w}$'s arising in the right-hand side of~\eqref{eq:phase-matching} is~$2^{\odima{4s}} N_s$, where~$N_s := {2n+1 \choose {\odima{2s}}}$. 
As such, by combining~\eqref{eq:detection-II-term2-one-corner}--\eqref{eq:phase-matching} and invoking the Krein-Milman theorem~(\cite[Ch.~8]{simon2011convexity}), we get
\begin{align}
\label{eq:rand-convo-term-bound}
2\Re \langle \xi, \wh \vphi^{(0)} * \xi \rangle_n 
&\le 4\sqrt{2} \left(3s + 2\log(e^3 s)\log\left(2^{4s} N_s\delta^{-1}\right) \right) \notag\\
&\le \odima{8}\sqrt{2} \log(e^4 s)\left(6s + \log(N_s\delta^{-1})\right)
\end{align}
with probability at least~$1-\delta$. Repeating the above argument for~$\wh T_n^{(k,\pm)}$, taking note of~\eqref{eq:detection-I-edge}, and plugging the result in~\eqref{eq:detection-I-split}, we obtain the following chain of inequalities with prob.~$\ge 1-\delta$:
\begin{align}
\wh T_n(y) 
&\le 
\odima{16}(K+1)\sqrt{2} \sigma^2  \log(e^4 s) \left(\odima{6s} +\log(2(K+1)N_s\delta^{-1})\right) \notag\\
&\le 
16(K+1) \sqrt{2} \sigma^2  \log(e^4 s) \left(6s + \odima{2s}\log(2n+1) + \log(2(K+1)\delta^{-1}) \right)\\
&\le \odima{\frac{48\sqrt{2}}{300}\sL_0^2}
< \odima{\frac{1}{4}}\sL_0^2.
\tag*{\qedhere}
\end{align}
We are done since, from the Type II error analysis,~$\wh \T_n(y) \ge \odima{\frac{3}{8}} \sL_0^2$ under~$\Hyp_{s,n}(\sL)$ w.p.~$\ge 1-\delta$.
\qed

\subsection{Proof of Theorem~\ref{th:l2con-core}}

\noindent\proofstep{1}.
We start with the same decomposition as in the proof of~\cite[Thm.~1]{harchaoui2019adaptive}, reminiscent of~\eqref{eq:detection-II-decomp}:
\begin{align}
\label{eq:l2con-error-decomp-1}
\|x - \wh\vphi * y \|_{n,2}^2 
&= 
\|y - \wh\vphi * y \|_{n,2}^2 - \sigma^2 \|\xi\|_{n,2}^2 - 2 \sigma \Re \langle \xi, x - \wh\vphi * y \rangle_n \notag\\
&\stackrel{}{\le}  \|y - \vphi^\star * y \|_{n,2}^2  - \sigma^2 \|\xi\|_{n,2}^2 - 2 \sigma \Re \langle \xi, x - \wh\vphi * y \rangle_n \notag\\
&= \|x - \vphi^\star * y \|_{n,2}^2 + 2 \sigma \Re \langle \xi, x - \vphi^\star * y \rangle_n - 2 \sigma \Re \langle \xi, x - \wh\vphi * y \rangle_n \notag\\
&\stackrel{}{=} \sigma^2\|\vphi^\star * \xi \|_{n,2}^2 - 2 \sigma^2 \Re \langle \xi, \vphi^\star * \xi \rangle_n - 2 \sigma \Re \langle \xi, x - \wh\vphi * y \rangle_n.
\end{align}
Here, the inequality holds due to~$\vphi^\star$ being feasible in~\eqref{opt:l2con-core}; the last identity is by the reproducing property of~$\vphi^\star$.
Now, for the first and second term in the right-hand side of~\eqref{eq:l2con-error-decomp-1} we can reuse, respectively, the estimates~\eqref{eq:detection-II-term1} and~\eqref{eq:detection-II-term2}.
On the other hand, let~$T_n: \C(\Z) \to \C_n(\Z)$ be the truncation operator that zeroes out~$x_\tau$ whenever~$\tau \notin [n]_{\pm}$; finally, let~$\Pi_n: \C_n(\Z) \to \C_n(\Z)$ be the Euclidean projector on the subspace~$X_n := T_n(X)$ of~$\C_n(\Z)$. 
%
Then for the last term in~\eqref{eq:l2con-error-decomp-1},
\begin{align}
\label{eq:l2con-error-decomp-2}
-2 \sigma \Re &\langle \xi, x - \wh\vphi * y \rangle_n \notag\\
&\stackrel{\hphantom{(a)}}{=}
2 \sigma^2 \Re \langle \xi, \wh\vphi * \xi \rangle_n - 2 \sigma \Re \langle T_n \xi, T_n [x - \wh\vphi * x] \rangle \notag\\
&\stackrel{(a)}{=}
2 \sigma^2 \Re \langle \xi, \wh\vphi * \xi \rangle_n - 2 \sigma \Re \langle \Pi_n T_n \xi, T_n [x - \wh\vphi * x] \rangle \notag\\
&\stackrel{(b)}{\le}
2 \sigma^2 \Re \langle \xi, \wh\vphi * \xi \rangle_n + 4\sigma^2 \| \Pi_n T_n \xi \|_{2}^2 + \tfrac{1}{4}\| x - \wh\vphi * x \|_{n,2}^2 \notag\\
&\stackrel{(c)}{\le}
2 \sigma^2 \Re \langle \xi, \wh\vphi * \xi \rangle_n + 4\sigma^2 \| \Pi_n T_n \xi \|_{2}^2 + \tfrac{1}{2} \| x - \wh\vphi * y \|_{n,2}^2 + \tfrac{\sigma^2}{2} \| \wh\vphi * \xi\|_{n,2}^2. \notag\\
&\stackrel{(d)}{\le}
2 \sigma^2 \Re \langle \xi, \wh\vphi * \xi \rangle_n + 4\sigma^2 \| \Pi_n T_n \xi \|_{2}^2 + \tfrac{1}{2} \| x - \wh\vphi * y \|_{n,2}^2 + \tfrac{\sigma^2}{2} \| \wh\vphi * \xi \|_{n,2}^2. 
\end{align}
Here, in~$(a)$ we used that~$x - \wh \vphi * x = x - \sum_{k \in [n]_{\pm}} \wh \vphi_k \Delta^k x \in X$ by the shift-invariance of~$X$, thus~$T_n [x - \wh \vphi * x] \in X_n$; 
in~$(b)$ and~$(c)$ we used Cauchy-Schwarz; in~$(d)$ we used nonexpansiveness of projection.
Combining~\eqref{eq:l2con-error-decomp-1}--\eqref{eq:l2con-error-decomp-2} we arrive at the decomposition
\begin{equation}
\begin{aligned}
\half\|x - \wh\vphi * y \|_{n,2}^2  
&\le \sigma^2\|\vphi^\star * \xi \|_{n,2}^2 - 2 \sigma^2 \Re \langle \xi, \vphi^\star * \xi \rangle_n 
	+ 2 \sigma^2 \Re \langle \xi, \wh\vphi * \xi \rangle_n + 4\sigma^2 \| \Pi_n T_n \xi \|_{2}^2 \\
	&\quad+ \tfrac{\sigma^2}{2} \| \wh\vphi * \xi \|_{n,2}^2.
\end{aligned}
\label{eq:l2con-error-decomp-total}
\end{equation}
Observe that in the right-hand side of~\eqref{eq:l2con-error-decomp-total}, the first three terms have already been estimated in the proof of Theorem~\ref{th:detection}; cf.~\eqref{eq:detection-II-term1},~\eqref{eq:detection-II-term2},~\eqref{eq:rand-convo-term-bound}, and the last term is easy to estimate: since~$\| \Pi_n T_n \xi \|_{2}^2 \sim \chi_{2 \dim(X_n)}^2$ where~$\dim(X_n) \le s$, by~\eqref{eq:quad-form-indef} we get that, w.p.~$\ge 1-\delta$,
\begin{equation}
\label{eq:l2con-chi-square}
\| \Pi_n T_n \xi \|_{2}^2 
\le 2s + 2\sqrt{s\log(\delta^{-1})} + \log(\delta^{-1})
\le 3s + 2\log(\delta^{-1}).
\end{equation}

\noindent\proofstep{2}.
To estimate the last term in~\eqref{eq:l2con-error-decomp-total}, we take a route similar to the one pursued to deal with the term~$\lang \xi, \wh \vphi * \xi \rang_n$ in the Type I error analysis of Theorem~\ref{th:detection}. 
By the Krein-Milman theorem,
\begin{align}
\label{eq:krein-milman-quad}
\| \wh\vphi * \xi \|_{n,2}^2
\le \sup_{\vphi \in \Ext_{s,n}} \| \vphi * \xi \|_{n,2}^2
\end{align}
where~$\Ext_{s,n}$, defined in~\eqref{eq:extreme-points-l1-linf}, is the set containing all extremal points of the feasible set in~\eqref{opt:l2con-core}.
Now, fix~$\vphi_{\cJ, \vec z} \in \Ext_{s,n}$ corresponding to some selection of indices~$\cJ := \{j_1, ..., j_{\odima{2s}}\} \subset \{0,1, ..., 2n\}$ and phase factors~$\vec z := (z_1, ..., z_{\odima{2s}}) \in \T^{\odima{2s}}$.
Following~\eqref{eq:convo-matrix}--\eqref{eq:Q-matrix} we note that, for any~$\vphi \in \C_n(\Z)$,
\begin{align}
\| \vphi * \xi \|_{n,2}^2 
= \| \Filt_{n,n}^{\vphantom\htop}(\vphi) \{\xi\}_{2n}^{\vphantom\htop} \|_{2}^2 
\le \| \Circ_{n,n}^{\vphantom\htop}(\vphi) \{\xi\}_{2n}^{\vphantom\htop} \|_{2}^2 
= (4n+1)\,\| \Diag(\F_{2n}[\vphi]) \F_{2n}[\xi] \|_{2}^2. 
\label{eq:conv-to-circ}
\end{align}
Furthermore, for~$\vphi = \F_n^{-1}[z e^j]$ with~$z \in \T$ and~$j \in \{0, 1, ..., 2n\}$, the entries of~$\F_{2n}[\vphi]$ write as
\[
(\F_{2n}^{\vphantom{-1}}[\F_n^{-1}[z e^j]])_k
= \frac{z}{\sqrt{(4n+1)(2n+1)}} \Dir_{n}\left(\frac{\chi_{k,2n}}{\chi_{j,n}}\right),
\]
where~$\chi_{j,n} = \exp\big(\frac{i2\pi j}{2n+1}\big)$ and~$\Dir_n(z) = \sum_{\tau \in [n]_{\pm}} z^\tau$ is the Dirichlet kernel (see Appendix~\ref{apx:dirichlet-and-fejer} for some background and auxiliary results on the Dirichlet and Fej\'er kernels). 
As the result, we get
\begin{align}
&(4n+1)\,\| \Diag(\F_{2n}[\vphi_{\cJ,\vec z}]) \F_{2n}[\xi] \|_{2}^2 \notag\\
&\le
	\frac{\odima{1}}{(2n+1)^2} \sum_{k \in [2n]_{\pm}} 
	|(\F_{2n}[\xi])_k|^2
	\Bigg( \sum_{j \in \cJ} \left| \Dir_{n}\left(\frac{\chi_{k,2n}}{\chi_{j,n}}\right) \right| \Bigg)^2 
= \F_{2n}[\xi]^\htop \Lambda_\cJ	\F_{2n}[\xi]
\end{align}
where~$\Lambda_{\cJ}$ is a diagonal matrix with positive entries on the main diagonal. 
Now: by Lemma~\ref{lem:oversampling-linf-sparse},
\begin{equation}
\label{eq:Lambda-op}
\|\Lambda_{\cJ} \|_{\op}
=
\frac{\odima{1}}{(2n+1)^2} \max_{w \in \T_{2n}} \Bigg(\sum_{j \in \cJ} \left| \Dir_{n}\left(\frac{w}{\chi_{j,n}}\right) \right| \Bigg)^2 
\le \log^2(e^3\odima{s}).
\end{equation}
This implies the following estimate for the trace of~$\Lambda_{\cJ}$:
\begin{align}
\label{eq:Lambda-trace}
\tr(\Lambda_{\cJ}) 
&\;\le 
\vspace{-0.3cm}\frac{\odima{1}}{(2n+1)^2} \Bigg( \max_{w' \in \T_{2n}} \sum_{j' \in \cJ} \left| \Dir_{n}\left(\frac{w'}{\chi_{j',n}} \right) \right| \Bigg)
	\sum_{j \in \cJ} \sum_{w \in \T_{2n}} \left| \Dir_{n}\left(\frac{w}{\chi_{j,n}}\right) \right|  \notag\\
\vspace{-0.3cm}
&\stackrel{\eqref{eq:Lambda-op}}{\le} 
\frac{\log(e^3\odima{s})}{(2n+1)}
	\sum_{j \in \cJ} \sum_{w \in \T_{2n}} \left| \Dir_{n}\left(\frac{w}{\chi_{j,n}}\right) \right|  \notag\\
&\;\le \frac{\odima{2s}\log(e^3s)}{(2n+1)} \max_{z \in \T_n} \sum_{w \in \T_{2n}} \left| \Dir_{n}\left(\frac{w}{z}\right) \right|  
\le 2s \log(e^3s) \log(e^3n).
\end{align}
where in the final step we applied Lemma~\ref{lem:oversampling-linf}. 
When combined together,~\eqref{eq:Lambda-op}--\eqref{eq:Lambda-trace} imply that
\begin{equation}
\label{eq:Lambda-frob}
\|\Lambda_{\cJ}\|_{\sF}^{\odima{2}} 
\le  \odima{2s \log(e^3\odima{s})^3 \log(e^3n)}.
\end{equation}
When combined with~\eqref{eq:quad-form-indef} and the union bound, the estimates~\eqref{eq:krein-milman-quad}--\eqref{eq:Lambda-frob} imply that w.p.~$\ge 1-\delta$, 
\begin{align}
\tfrac{1}{2} \| \wh\vphi * \xi \|_{n,2}^2 
&\le \max_{\cJ \,\subset\, \{0,1, ..., 2n\}: \; |\cJ| = \odima{2s}} \tr(\Lambda_{\cJ}) + \| \Lambda_{\cJ} \|_{\sF} \sqrt{2 \log(N_s\delta^{-1})} + \|\Lambda_{\cJ}\|_{\op} \log(N_s\delta^{-1}) \notag\\
&\le \odima{2s}\log(e^3\odima{s}) \log(e^3n) + \odima{2} \log^{3/2}(e^3\odima{s})\sqrt{s \log(e^3n) \log(N_s\delta^{-1})} + \log^2(e^3\odima{s}) \log(N_s\delta^{-1}) \notag\\
&\le (\odima{3s}\log(e^3s) \log(e^3n) + \odima{2}\log^2(e^3s) \log(N_s\delta^{-1})) \notag\\
&\le (\odima{7s} \log(e^3n) + \odima{2}\log(\delta^{-1})) \log^2(e^3s),
\label{eq:l2con-convo-quad-bound}
\end{align}
where the penultimate step is by Cauchy-Schwarz; in the last we used~$\log N_s \le \odima{2s}\log(2n+1)$.

\noindent
Plugging~\eqref{eq:detection-II-term1},~\eqref{eq:detection-II-term2},~\eqref{eq:rand-convo-term-bound},~\eqref{eq:l2con-chi-square},~\eqref{eq:l2con-convo-quad-bound} in~\eqref{eq:l2con-error-decomp-total}, and using~$\sqrt{2} \le \frac{3}{2}$, by simple algebra we get~\eqref{eq:l2-core}.
\qed


\subsection{Proof of Corollary~\ref{cor:l2-full}}
In the setup of Corollary~\ref{cor:l2-full}, i.e.~assuming~$n,s$ are powers of 3, we make use of the decomposition
\[
\big\|\wh x^\full - x\big\|_{2n,2}^2 = \big\|\wh x^\full - x\big\|_{n,2}^2 + \left( \sum_{k = 1}^K \left\|\Delta^{\pm h_k}\big[\wh x^\full - x\big]\right\|_{n_k,2}^2 \right) + \sigma^2 \left( \|\xi\|_{2n,2}^2 - \|\xi\|_{2n-\odima{3s},2}^2\right).
\]
Here, there are~$2K+2 \le 2\log_3(n/s)$ terms; each but the last one admits the bound of Theorem~\ref{th:l2con-core}, and for the last one we use the chi-square bound~\eqref{eq:detection-I-edge}. The result of Corollary~\ref{cor:l2-full} now follows via the union bound \odima{and some arithmetics to account for the standalone chi-squared term.}
\qed

\appendix

\section{Auxiliary results}
\label{apx:aux}

\subsection{Bounding the norm of a convolution power}
\label{apx:convolution-powers}


\begin{proposition}[Convolution powers]
\label{prop:convolution-powers}
If~$X \subseteq \C(\Z)$ admits a reproducing filter~$\phi \in  \C_m(\Z)$ with~$\| \phi \|_2 \le \frac{\sR}{\sqrt{2m+1}}$, then
for any~$k \in \{2,3,...\}$,~$\phi^{k} \in \C_{km}(\Z)$ is also reproducing for~$X$, with
\[
\|\phi^{k}\|_2 \le \|\F_{km} [\phi^{k}]\|_1 \le \frac{c_k \sR^k}{\sqrt{2km+1}}
\]
where~$c_{k} = 2^{k-1}$ for~$k = 2^p$ and~$c_{2k+1} \le 3 \sqrt{2k+1}c_k^2$ for~$k \in \N$; as such,~$\log(c_k) = O(k\log k)$.
\end{proposition}
\begin{proof}
The reproducing property follows from~$1-\phi^k(z) = (1-\phi(z))(1+\phi(z)+...+\phi^{k-1}(z))$.
For~$k = 2^p$, we proceed by induction over~$p$: assume that~$\|\phi^{k}\|_2 \le \|\F_{km} [\phi^{k}]\|_1 \le \frac{c_{k} \sR^{k}}{\sqrt{2km+1}}$, then
\[
\begin{aligned}
\|\F_{2km} [\phi^{2k}]\|_1 
\stackrel{\eqref{def:DFT}}{=}  
\sqrt{4km+1} \|\F_{2km}[\phi^k] \|_2^2 
= \sqrt{4km+1} \|\phi^k\|_2^2 
\le
\frac{\sqrt{4km+1}\,c_{k}^2 \sR^{2k}}{2km+1}
\le 
\frac{2c_{k}^2 \sR^{2k}}{\sqrt{4km+1}}.
\end{aligned}
\]
Whence~$c_{2k}^{\vphantom 2} = 2c_{k}^2$ for all~$k \in \N$. 
Unrolling this recursion we arrive at~$c_{2^p} = 2^{\sum_{j = 0}^{p-1} 2^j} = 2^{2^p-1}$.
Now, to bound~$c_{2k+1}$ in terms of~$c_{k}$ observe that
\[
\begin{aligned}
\|\F_{(2k+1)m} [\phi^{k}]\|_1 
&\stackrel{\eqref{def:DFT}}{=}  
\frac{1}{\sqrt{2(2k+1)m+1}} \sum_{z \in \T_{(2k+1)m}} \left| \phi^{2k+1}(z) \right| \\
&\le
(2(2k+1)m+1)\,\| \F_{(2k+1)m}[\phi] \|_{\infty} \|\F_{(2k+1)m}[\phi^k] \|_2^2 \\
&\le 
(2(2k+1)m+1)\,\| \phi\|_2^{\vphantom 2} \,\|\phi^k\|_2^2 \\
&\le 
\frac{2(2k+1)m+1}{2km+1} \, \frac{c_{k}^2 \sR^{2k+1}}{\sqrt{2m+1}} 
\le 
\frac{2k+1}{k} \, \frac{c_{k}^2 \sR^{2k+1}}{\sqrt{2m+1}} 
\le 
\frac{3\sqrt{2k+1}c_{k}^2 \sR^{2k+1}}{\sqrt{2(2k+1)m+1}}.
\end{aligned}
\]
As such,~$c_{2k+1}^{\vphantom 2} \le 3\sqrt{2k+1} c_{k}^2$, indeed.
Finally, to verify that~$\log(c_k) = O(k \log k)$ for all~$k \ge 2$, it suffices to consider the binary representation of~$k$, and alternate between the two cases.
\end{proof}


\subsection{Extreme points of the ball of the norm~$\max\{\|\cdot\|_1, s\|\cdot\|_\infty\}$ on~$\C^n$}
\label{apx:extreme-points}
In what follows, we let~$n,s \in \N \cup \{0\}$ and define the sets~$\Comp^n(s), \Real^n(s), \Posi^{n}(s)$ as follows:
\begin{equation}
\label{def:cap-balls}
\bald
\Comp^n(s)
	&:= \{w \in \C^{n}: \;  \| w \|_{\infty} \le 1, \, \|w\|_1 \le s\}, \\
\Real^n(s)
	&:= \{u \in \R^{n}: \;  \| u \|_{\infty} \le 1, \, \|u\|_1 \le s\}, \\
\Posi^{n}(s)
	&:= \{u \in \R_+^{n}: \;  \| u \|_{\infty} \le 1, \, \|u\|_1 \le s\}.
\eald
\end{equation}
We first recall the (known) explicit characterizations of the vertices of~$\Posi^n(s)$ and~$\Real^n(s)$.
\begin{proposition}[{\cite[Thm.~2.1 and Eq.~(2.1)]{deza2021polytopal}}]
\label{prop:cap-real}
The polytope~$\Posi^n(s)$ (resp.~$\Real^n(s)$) has~${n \choose s}$ $ $  $ $(resp.~$2^s {n \choose s}$) vertices, namely all~$s$-sparse vectors with entries in~$\{0,1\}$ (resp.~$\{-1, 0, 1\}$). 
\end{proposition}
As an immediate corollary, we restrict the possible structure of extremal points for~$\Comp^n(s)$. 
\begin{proposition}
\label{prop:cap-comp}
Let~$e^{0}, ..., e^{n-1}$ be the canonical basis of~$\R^{n}$.
Any extremal point of~$\Comp^n(s)$ is of the form~$\sum_{j = 0}^{n-1} z_j u_j e^j$ for some~$z_0, ..., z_{n-1} \in \T$ and~$s$-sparse binary vector~$u = (u_0,...,u_{n-1})$.
\end{proposition}
\begin{proof}
Let~$w \in \Comp^n(s)$, then~$|w| \in \Posi^n(s)$ for~$|w| = (|w_0|, ..., |w_{n-1}|)$.
By Proposition~\ref{prop:cap-real}, there exist a sequence~$u^{1}, u^{2}, ...$ of~$s$-sparse binary vectors and a sequence~$\alpha_1, \alpha_2, ...$ of nonnegative coefficients adding up to 1, such that
$
|w_j| = \sum_{k} \alpha_k (u^{k})_j
$
for~$0 \le j \le n-1$. Multiplying this by~$z_j := w_j / |w_j| \in \T$ we get
$w_j = \sum_{k} \alpha_k z_j (u^{k})_j$ 
for~$0 \le j \le n-1$.
In other words, any~$w \in \Comp^n(s)$ is a convex combination of vectors of the form~$( z_0 u_0, ..., z_{n-1} u_{n-1})$, where each~$u = (u_0, ..., u_{n-1})$ is an~$s$-sparse binary vector.
\end{proof}

\section{Technical results on trigonometric interpolation}
\label{apx:dirichlet-and-fejer}

\subsection{Dirichlet and Fej\'er kernels}

The Dirichlet and Fej\'er kernels, respectively, are the following Laurent polynomials~\cite{zygmund2002trigonometric}:
\begin{align}
\label{def:Dir-and-Fej}
\Dir_m(z) := \sum_{|k| \le m} z^k, \quad
\Fej_m(z) 
&:= \sum_{|k| \le m} \left(1-\frac{|k|}{m+1}\right)z^k.
\end{align}
That is,~$\Dir_m(z)$ is the~$z$-transform of the rectangular window signal~$u_t = \ones\{|t| \le m\}$;~$\Fej_{m}(z)$ that of the triangular window signal~$v_t (1-\frac{|t|}{m+1}) \ones\{|t| \le m\}$, 
and~$\Fej_{2m}(z) = \frac{1}{\odima{2}m+1} \Dir_m^2(z)$.
Note that~$\frac{1}{2m+1}\Dir_m(1) = \frac{1}{m+1}\Fej_m(1) = 1$. 
These kernels can be recast as trigonometric polynomials:
\[
\dir_m(\omega) := \Dir_m(e^{i\omega}) = \frac{\sin((2m+1)\frac{\omega}{2})}{\sin(\frac{\omega}{2})},
\quad
\fej_m(\omega) := \Fej_m(e^{i\omega}) = \frac{1}{m+1} \left(\frac{1-\cos((m+1)\omega)}{1-\cos(\omega)} \right).
\]
Finally, we define the causal counterpart of the Fej\'er kernel by appropriately shifting~$\Fej_m(z)$:
\begin{equation}
\label{def:Fej-causal}
\Fej_{2m}^+(z) := z^m \Fej_m(z).
\end{equation}
Note that~$\Fej_{2m}^+ \in \C_{2m}^+(\Z)$. Moreover,~$\Fej_{2m}^+(1) = m+1$, and~$|\Fej_{2m}^+(z)| = |\Fej_m(z)|$ for all~$z \in \T$.

\paragraph{Summation over equidistant grids.}
We bound the sums of~$|\Dir_n^{\vphantom+}|$,~$|\Fej_n^{\vphantom+}|$,~$|\Fej_{2n}^+|$ over~$\T_N$.
\begin{lemma}
\label{lem:oversampling}
For any~$n,N \in \N$ and~$a \in \T$, letting~$H_N$ be the~$N^{\textup{th}}$ harmonic number, one has
\[
\begin{aligned}
\frac{1}{2N+1}\sum_{w \in \T_N} \left|\Dir_{n}\Big(\frac{w}{a}\Big)\right|
&\le \frac{4n+2}{2N+1} + H_N,\\
\frac{1}{2N+1}\sum_{w \in \T_N} \left|\Fej_{2n}^+\Big(\frac{w}{a}\Big)\right|
= \frac{1}{2N+1}\sum_{w \in \T_N} \left|\Fej_{n}\Big(\frac{w}{a}\Big)\right|
&\le \frac{2n+2}{2N+1} + \left( \frac{2N+1}{2n+2} \right) \frac{\pi^2}{6}.
\end{aligned}
\]
\end{lemma}
\begin{proof}
Since~$|\Fej_{2n}(z)| = |\Fej_n^+(z)|$ for~$z \in \T$, the summation of~$|\Fej_{2n}^+|$ reduces to that of~$|\Fej_{n}|$. 
Now, for any~$a \in \C$, at most 2 points of the grid~$\{\frac{w}{a}: w \in \T_N\}$ have argument in~$(-\frac{2\pi}{2N+1},\frac{2\pi}{2N+1})$.
Each of them contributes at most~$\Dir_n(1) = 2n+1$ or~$\Fej_n(1) = n+1$ to the respective sum. 
As for the remaining points, we use the following elementary estimates valid for all~$\omega \in [-\pi,\pi]$,
\vspace{-0.1cm}\[
|\dir_n(\omega)| \le \frac{1}{|\sin(\frac{\omega}{2})|} \le \frac{\pi}{|\omega|}
\quad\text{and}\quad
0 \le (n+1)\fej_n(\omega) \le \frac{1}{\sin^2(\frac{\omega}{2})} \le \frac{\pi^2}{\omega^2},
\]
whose right-hand sides decrease in~$|\omega|$, and summate separately over at most~$N$ points with argument in~$[\frac{2\pi}{2N+1}, \pi]$ and at most~$N$ points with argument in~$[-\pi,-\frac{2\pi}{2N+1}]$. 
Overall, this gives
\[
\begin{aligned}
\sum_{w \in \T_N} \left|\Dir_{n} \hspace{-0.1cm}\left(\frac{w}{a}\right)\right|
&\le 2\left(2n+1 + \sum_{k = 1}^N \frac{\pi}{\arg(\chi_{k,N})} \hspace{-0.1cm}\right)
= {4n+2} + (2N+1)\sum_{k = 1}^N \frac{1}{k},\\
\sum_{w \in \T_N} \left|\Fej_{n}\hspace{-0.1cm}\left(\frac{w}{a}\right)\right|
&\le 2\left(n+1 +  \frac{1}{n+1} \sum_{k = 1}^N \frac{\pi^2}{\arg(\chi_{k,N})^2} \hspace{-0.1cm}\right) 
= {2n+2} + \frac{(2N+1)^2}{2n+2} \sum_{k = 1}^N \frac{1}{k^2}.
\end{aligned}
\]
Here we used that~$\arg(\chi_{k,N}) = \frac{2\pi k}{2N+1}$ for~$k \in \{1,...,N\}$;~cf.~\eqref{def:DFT-grid}. 
\end{proof}

\paragraph{Oversampling bounds.} 
As it was observed in~\cite{harchaoui2019adaptive}, the Dirichlet kernel summation estimate of Lemma~\ref{lem:oversampling} directly implies a uniform estimate for the inflation of~$\ell_1$-norm with oversampling, that is~the uniform over~$\vphi \in \C_n(\Z)$ upper bound for the ratio~$\|\F_{N}[\vphi]\|_{1}/\|\F_{n}[\vphi]\|_{1}$ with~$N \ge n$. 

\begin{lemma}[Oversampling in~$\ell_1$]
\label{lem:oversampling-l1}
For any~$n,N \in \N$ such that~$N \ge n$, and nonzero~$\vphi \in \C_n(\Z)$,
\[
\begin{aligned}
\frac{\| \F_{N}[\vphi] \|_1}{ \| \F_{n}[\vphi] \|_1} 
&\le \frac{1}{\sqrt{(2N+1)(2n+1)}} \max_{a \in \T_n} \sum_{w \in \T_{N}} \left|\Dir_{n}\left(\frac{w}{a}\right)\right| 
\le \sqrt{\frac{2N+1}{2n+1}} \log(eN) + 2\sqrt{\frac{2n+1}{2N+1}}.
\end{aligned}
\]
\begin{proof}
We first note that the convex function~$\| \F_{N}^{\vphantom{\,-1}}[\F_{n}^{\,-1}[\cdot]]\|_1$ is maximized at an extreme point of the unit~$\ell_1$-ball in~$\C^{2n+1}$, i.e.~on a vector of the form
$
z e^{j},
$
where~$z \in \T$ and~$e^{j}$,~$0 \le j \le 2n$, is the~$j^\textup{th}$ canonical basis vector of~$\R^{2n+1}$. 
As the result,
\[
\bald
\max_{\vphi \in \C_n(\Z) \setminus 0} \frac{\| \F_{N}[\vphi] \|_1}{ \| \F_{n}[\vphi] \|_1} 
&=  \max_{0 \le j \le 2n} \| \F_{N}^{\vphantom{\,-1}}[\F_n^{\,-1}[e^j]] \|_1 
= \frac{1}{\sqrt{(2N+1)(2n+1)}} \max_{0 \le j \le 2n} \sum_{k = 0}^{2N} \left|\Dir_{n}\left(\frac{\chi_{k,N}}{\chi_{j,n}}\right)\right|.
\eald
\]
After that, it only remains to invoke Lemma~\ref{lem:oversampling}. 
\end{proof}
\end{lemma}
\begin{remark}
\label{rem:dirichlet-log}
{\em For~$N \ge (1+c)n$, the inequality of Lemma~\ref{lem:oversampling-l1} is sharp up to a constant factor.}
\end{remark}
Below we give a counterpart result for~$\ell_\infty$-norm. 
In the context of trigonometric interpolation, the corresponding ratio is known as the Lebesgue constant, and a sharp estimate is known for it (e.g.~\cite{ehlich1966auswertung},~\cite[Thm.~2.1]{sorevik2016trigonometric}).
We give a standalone estimate to keep the paper self-contained.

\begin{lemma}[Oversampling in~$\ell_\infty$]
\label{lem:oversampling-linf}
For any~$n,N \in \N$ with~$N \ge n$, and nonzero~$\vphi \in \C_n(\Z)$, 
\[
\begin{aligned}
\frac{\| \F_{N}[\vphi] \|_\infty}{ \| \F_{n}[\vphi] \|_\infty} 
\le 
	\frac{1}{\sqrt{(2N+1)(2n+1)}} \max_{a \in \T_N} \sum_{w \in \T_n} \left| \Dir_{n} \left( \frac{w}{a} \right) \right|
	\le \sqrt{\frac{2n+1}{2N+1}} (H_n + 2).
\end{aligned}
\]
As a consequence, for any~$\vphi \in \C_n(\Z)$ it holds that\,~$\max_{z \in \T} |\vphi(z)| \le (H_n + 2) \, \max_{z \in \T_n} |\vphi(z)|$.
\end{lemma}
\begin{proof}
The convex function~$\| \F_{N}^{\vphantom{\,-1}}[\F_{n}^{\,-1}[\cdot]]\|_\infty$ is maximized at an extreme point of the unit~$\ell_\infty$-ball in~$\C^{2n+1}$, i.e.~on a vector of the form~$(z_0, z_1, ..., z_{2n})$
where~$z_j \in \T$ for~$0 \le j \le 2n$.
As the result,
\[
\bald
\max_{\vphi \in \C_n(\Z) \setminus 0} \frac{\| \F_{N}[\vphi] \|_\infty}{ \| \F_{n}[\vphi] \|_\infty} 
&=  
	\max_{z_0, ..., z_{2n} \in \T} \big\| \F_{N}^{\vphantom{\,-1}} \big[\F_n^{\,-1} \big[ z_0 e^0 + z_1 e^1 + ... + z_{2n} e^{2n} \big] \big] \big\|_\infty \\
&= 
	\frac{1}{\sqrt{(2N+1)(2n+1)}} \max_{z_0, ..., z_{2n} \in \T} \max_{0 \le k \le 2N} \left| \sum_{j = 0}^{2n} z_j\Dir_{n} \left( \frac{\chi_{j,n}}{\chi_{k,N}} \right) \right| \\
&= 
	\frac{1}{\sqrt{(2N+1)(2n+1)}} \max_{0 \le k \le 2N} \sum_{j = 0}^{2n} \left| \Dir_{n} \left( \frac{\chi_{j,n}}{\chi_{k,N}} \right) \right|.
\eald
\]
As previously, we conclude by invoking Lemma~\ref{lem:oversampling}. 
\end{proof}

We are now about to enhance the previous result: it turns out that the worst-case ratio of~$\ell_\infty$-norms reduces to~$O(\log(s))$ if we only allow for~$s$-sparse vectors in the spectral domain.

\begin{lemma}[Oversampling in~$\ell_\infty$ under sparsity]
\label{lem:oversampling-linf-sparse}
For~$s,n,N \in \N$ such that~$N \ge n$, one has
\[
\bald
\max_{w \in \Comp^{2n+1}(s)} \| \F_{N}^{\vphantom{\,-1}}[\F_n^{\,-1}[w]] \|_\infty
\le \sqrt{\frac{2n+1}{2N+1}} (\log(\lceil s/2 \rceil) + 3).
\eald
\]
\end{lemma}

\begin{proof}

We can assume that~$s \le 2n$, as otherwise the claim reduces to Lemma~\ref{lem:oversampling-linf}. 
The convex function~$\| \F_{N}^{\vphantom{\,-1}}[\F_{n}^{\,-1}[\cdot]]\|_\infty$ is maximized at an extreme point of~$\Comp^{2n+1}(s)$. By Proposition~\ref{prop:cap-comp}, these are of the form~$z_1 e^{j_1} + ... z_s e^{j_s}$ for some~$z_1, ..., z_s \in \T$ and~$\{j_1, ...,  j_s\} \subset \{0, 1, ..., 2n\}$. 
Thus,
\begin{align}
&\max_{w \in \Comp^{2n+1}(s)} \| \F_{N}^{\vphantom{\,-1}}[\F_n^{\,-1}[w]] \|_\infty
=  
	\max_{
	\scriptsize
	\begin{aligned}
		\{j_1, ...,  j_s\} &\subset \{0, 1, ..., 2n\},\notag\\
		z_1, ..., z_{s}    &\in \T
	\end{aligned}
	} \; \big\| \F_{N}^{\vphantom{\,-1}} \big[\F_n^{\,-1} \big[ z_1 e^{j_1} + ... + z_{s} e^{j_s} \big] \big] \big\|_\infty \\
&= 
	\frac{1}{\sqrt{(2N+1)(2n+1)}} 
	\max_{
	\scriptsize
	\begin{aligned}
		\{j_1, ...,  j_s\} &\subset \{0, 1, ..., 2n\},\notag\\
		z_1, ..., z_{s}    &\in \T
	\end{aligned}
	} \; \max_{0 \le k \le 2N} \left| z_1 \Dir_{n} \hspace{-0.1cm}\left( \frac{\chi_{j_1,n}}{\chi_{k,N}} \right) + ... + z_s \Dir_{n} \hspace{-0.1cm}\left( \frac{\chi_{j_s,n}}{\chi_{k,N}} \right)\right| \\
&= 
	\frac{1}{\sqrt{(2N+1)(2n+1)}} 
	\max_{\{j_1, ...,  j_s\} \subset \{0, 1, ..., 2n\}} \,  \max_{0 \le k \le 2N} \left( \left| \Dir_{n} \hspace{-0.1cm}\left( \frac{\chi_{j_1,n}}{\chi_{k,N}} \right) \right| + ... + \left| \Dir_{n} \hspace{-0.1cm}\left( \frac{\chi_{j_s,n}}{\chi_{k,N}} \right) \right| \right) \notag\\
&\le 
\hspace{-0.05cm}\sqrt{\frac{2n+1}{2N+1}} (H_{\lceil s/2 \rceil} + 2).
\end{align}
In the end, we proceeded as in the proof of Lemma~\ref{lem:oversampling} using the bound~$|\dir_n(\omega)| \le {\pi}{|\omega|^{-1}}$.
\end{proof}

\section{Technical results concerning shift-invariant subspaces of~$\C(\Z)$}
\label{apx:shift-inv}

\subsection{Characterization of finite-dimensional shift-invariant subspaces of~$\C(\Z)$}
\label{sec:SIS-characterization}

As shown in~\cite{harchaoui2019adaptive},~$s$-dimensional shift-invariant subspaces of~$\C(\Z)$ can be characterized in terms of the corresponding characteristic polynomial.\footnote{Characterization of the invariant subspaces for the restriction of~$\Delta$ to~$\ell^2$ is a classical topic~\cite{beurling1949two,halmos1961shifts,nikol'skii1967invariant}. Yet, these classical results seem to be limited to Hilbert spaces, and therefore not allowing to recover Proposition~\ref{prop:SIS-characterization}.}
To state the corresponding result, we start with a brief reminder.
The {\em lag} operator~$\Delta: \C(\Z) \to \C(\Z)$ is defined by~$(\Delta x)_t = x_{t-1}$.
A linear subspace~$X$ of~$\C(\Z)$ is called {\em shift-invariant} if it is an invariant subspace of~$\Delta$, that is~$\Delta X \subseteq X$.
An easy exercise is to verify that~$\Delta$ is bijective on any shift-invariant subspace, thus~$\Delta X = X$.
\begin{proposition}[{\cite[Proposition~5]{harchaoui2019adaptive}}]
\label{prop:SIS-characterization} 
Any shift-invariant subspace of $\C(\Z)$ of dimension~$s$ is the set of solutions for some homogeneous linear difference equation~$\sP(\Delta) x \equiv 0$ with a characteristic polynomial~$\sP(z) = \sP_0 + \sP_1 z + ... + \sP_s z^s$ of degree~$s$,  unique if normalized by~$\sP(0)=1$.
\end{proposition}

Recall that for~$w_1, ..., w_s \in \C$, we let~$\X(w_1,...,w_s)$ be the solution set of the equation~\eqref{eq:intro-ODE} whose characteristic polynomial is 
$
\sP(z) = \prod_{k \in [s]} \left(1-{w_k}{z}\right),
$
i.e.~has~$w_1^{-1}, ..., w_s^{-1}$ as its roots. By Proposition~\ref{prop:SIS-characterization}, any SIS of dimension~$s$ writes as~$X(w_1, ..., w_s)$ for a unique, up to a permutation, selection of~$(w_1, ..., w_s) \in \C^s$.
In the case of distinct roots,~$X(w_1, ..., w_s)$ is the span of~$s$ complex exponentials~$w_1^t,..., w_s^t$, i.e.~comprises signals of the form
$x_t = \sum_{k = 1}^s c_k^{\vphantom t} w_k^t$ with~$c_k \in \C.$
Generally, we get exponential polynomials: if the root~$w_{k}^{-1}$ of~$\sP$ has multiplicity~$m_k$, then~$X(w_1, ..., w_s)$ comprises signals of the form~$x_t = \sum_{k \in [s]} q_k(t)^{\vphantom t} w_{k}^t$, where~$\deg(q_k) = m_k-1$ and~$\sum_{k \in [s]} m_k = s$. 

\subsection{Minimal-norm one-sided reproducing filter for the polynomial subspace}
\label{apx:one-sided}

\begin{proposition}
\label{prop:causal-bound}
Let~$X$ be the space of complex polynomials of degree at most~$s-1$. 
Any sequence~$\{\phi_+^{*,m} \in \C_m^{+}(\Z)\}_{m \in \N}$ of one-sided reproducing for~$X$ filters of minimal~$\ell_2$-norm satisfies
\begin{equation}
\label{eq:causal-bound}
\lim_{m \to \infty} m\|\phi_+^{*,m}\|_{2}^2 =  s^2. 
\end{equation}
\end{proposition}
\begin{proof}
%
Fixing~$m \ge s$, we first focus on~$\psi \in \C_m^+(\Z)$ such that~$\psi(0) = 0$, denoted as~$\psi \in \C_m^{++}(\Z)$, 
and associate any such~$\psi$ with a vector~$(\psi_1, ..., \psi_m)$. We return to the general case in the end.

\noindent
\proofstep{1}. 
Observe that~$X$ is reproduced by~$\psi \in \C_m^{++}(\Z)$ if and only if~$\psi$ satisfies the equations
\begin{equation}
\label{eq:poly-linsys}
\begin{bmatrix}
1 & 1 & 1 & \dots & 1 \\
0 & \frac{1}{m} & \frac{2}{m} & \dots & 1 \\
\vdots & \vdots & \vdots & \ddots & \vdots \\
0 & \left(\frac{1}{m}\right)^{s-1} & \left(\frac{2}{m} \right)^{s-1} & \dots & 1
\end{bmatrix}
\begin{bmatrix}
-1\\
\psi
\end{bmatrix}
= 0.
\end{equation}
Indeed, each row in~\eqref{eq:poly-linsys} is the slice~$(x_0, x_1, ..., x_m)$ of~$x \in X$, so any reproducing~$\psi \in \C_m^{++}(\Z)$ satisfies~\eqref{eq:poly-linsys}. Conversely, the rows in~\eqref{eq:poly-linsys} slice a monomial basis of~$X$, so any solution to~\eqref{eq:poly-linsys} reproduces an arbitrary~$x \in X$ at~$t = 0$, but then at any~$t \in \Z$ as well, since~$X$ is shift-invariant. 

\noindent\proofstep{2}. 
Letting~$e_0$ be the first canonical basis vector of~$\R^s$,  we can rewrite~\eqref{eq:poly-linsys} as
\begin{equation}
\label{eq:poly-linsys-resolved}
\begin{bmatrix}
1 & 1 & \dots & 1 \\
\frac{1}{m} & \frac{2}{m} & \dots & 1 \\
\vdots & \vdots & \ddots & \vdots \\
\left(\frac{1}{m}\right)^{s-1} & \left(\frac{2}{m}\right)^{s-1} & \dots & 1
\end{bmatrix}
\psi 
= 
e_0,
\end{equation}
that is~$V^\top \psi = e_0$ where~$V = {V}_{s}(\frac{1}{m},\frac{2}{m},...,1) \in \R^{m \times s}$ is a full-rank Vandermonde matrix for the interpolation grid~$\{\frac{1}{m}, \frac{2}{m}, ..., 1 \}$.
The minimal-norm solution~$\psi^{*,m}_{++}$ of~\eqref{eq:poly-linsys-resolved} is unique and reads
\[
\psi^{*,m}_{+,+} = V (V^\top V)^{-1} e_0,
\]
so that~$m \|\psi^{*,m}_{+,+}\|_2^2 = [H_{s,m}^{-1}]_{0,0}$ where~$H_{s,m} := \frac{1}{m} V^\top V$. 
But for the entries of~$H_{s,m}$ we have
\[
\lim_{m \to \infty} [H_{s,m}]^{\vphantom{m}}_{j,k} 
= \lim_{m \to \infty} \frac{1}{m+1} \sum_{\tau = 1}^m \left(\frac{\tau}{m+1}\right)^{j+k} 
= \int_{0}^{1} u^{j+k} du
= \frac{1}{j+k+1}
= [H_s]_{j,k}
\]
where~$H_{s}$ is the~$s \times s$ Hilbert matrix (recall that~$j,k \in \{0, ..., s-1\}$ by our indexing convention).
Using the identity~$[H_{s}^{-1}]_{0,0} = s^2$, e.g.~\cite{choi1983tricks}, we arrive at the ``strictly one-sided" version of~\eqref{eq:causal-bound}:
\begin{equation}
\label{eq:strict-causal-bound}
\lim_{m \to \infty} m\|\psi^{*,m}_{++}\|_{2}^2 =  s^2.
\end{equation}
\noindent\proofstep{3}. 
Finally, observe that any reproducing for~$X$ filter~$\phi = (\phi_0, \phi_1, ..., \phi_m) \in \C_m^+(\Z)$ with~$\phi_0 \ne 1$ generates a reproducing for~$X$ filter~$\psi = \frac{1}{1-\phi_0} (0, \phi_1, ..., \phi_m) \in \C_m^{++}(\Z)$ whose squared norm is
\[
\|\psi\|_2^2 = \frac{\| \phi \|_2^2 - |\phi_0|^2}{|1-\phi_0|^2}.
\]
Choosing the minimal-norm~$\phi = \phi^{*,m}_{+}$ in the right-hand side, and comparing with~\eqref{eq:strict-causal-bound}, we get
\[
s^2 
= \lim_{m \to \infty} m\big\|\psi^{*,m}_{++}\big\|_2^2 
\le \liminf_{m \to \infty} m \left(\frac{\big\|\phi_+^{*,m}\big\|_2^2 - \big|[\phi_+^{*,m}]_0\big|^2}{\left|1-[\phi_+^{*,m}]_0\right|^2} \right) 
= \liminf_{m \to \infty}  m \left\|\phi_+^{*,m} \right\|_2^2,
\]
where~$\left|[\phi_+^{*,m}]_0\right| \to 0$ by~\eqref{eq:strict-causal-bound}.
Since
$
\|\psi^{*,m}_{++}\|_2 \ge \|\phi_+^{*,m}\|_2,
$
the squeeze theorem implies~\eqref{eq:causal-bound}.
\end{proof}

\section*{Acknowledgments}
The author is indebted to Profs.~Anatoli Juditsky and Arkadi Nemirovski for insightful discussions, encouragement to write this paper, \mbox{and the privilege of learning from them.}
He thanks Artem Zvavitch and Fedor Nazarov for organizing his visit at Kent State University in Spring 2026.

\bibliography{references}
\bibliographystyle{unsrt}

\end{document}